\documentclass{amsart}
 \usepackage{graphicx}
\usepackage{amssymb, color}\usepackage{amsmath}
\usepackage{amsfonts}
\usepackage{latexsym}
\makeatletter
 \def\LaTeX{\leavevmode L\raise.42ex
   \hbox{\kern-.3em\size{\sf@size}{0pt}\selectfont A}\kern-.15em\TeX}
\makeatother

\newcommand{\BibTeX}{{\rm B\kern-.05em{\sc
i\kern-.025emb}\kern-.08em\TeX}}
\newtheorem{col}{Corollary}[section]
\newtheorem{thm}{Theorem}[section]
\newtheorem{defn}{Definition}

\newtheorem{rem}[thm]{Remark}
\theoremstyle{defn}
\newtheorem{lem}[thm]{Lemma}

\numberwithin{equation}{section}

\def\EB{{\bf E}}
\def\FB{{\bf F}}
\def\AB{{\bf A}}
\def\HB{{\bf H}}

\begin{document}

\title[To multidimensional Mellin  analysis]{To Multidimensional Mellin  Analysis: Besov spaces, $K$-functor, approximations, frames.}
\author{Isaac Z. Pesenson}
\address{Department of Mathematics, Temple University,
Philadelphia, PA 19122} \email{pesenson@temple.edu}


\begin{abstract}

In the setting of the multidimensional Mellin analysis we introduce moduli of continuity and use them to define Besov-Mellin spaces. We prove that Besov-Mellin spaces are the interpolation spaces (in the sense of J.Peetre) between two Sobolev-Mellin spaces. We also introduce Bernstein-Mellin spaces and prove corresponding direct and inverse approximation theorems. In the Hilbert case we discuss Laplace-Mellin operator and define relevant Paley-Wiener-Mellin spaces. Also in the Hilbert case we describe Besov-Mellin spaces in terms of Hilbert frames.   

\end{abstract}

\dedicatory{Dedicated to 95th birthday of my teacher Paul L.Butzer}

\maketitle

\section{Introduction} 

By one-dimensional Mellin analysis we understand the following setup.
For $p \in  [1, \infty[$, denote by $\| \cdot \|_{p} $ the norm of the Lebesgue space $L ^{p}(\mathbb{R}_{+}).$ In Mellin analysis, the analogue of $L ^{p }(\mathbb{R}_{+} )$ are the spaces $X^{p}(\mathbb{R}_{+})$ comprising all functions $f: \mathbb{R}_{+} \mapsto \mathbb{ C}$  such that $f(\cdot)(\cdot)^{-1/p}\in  L ^{p}(\mathbb{R}_{+}) $ with the norm $\|f\|_{X^{p}(\mathbb{R}_{+}) }:= \| f (\cdot)(\cdot)^{-1/p}\|_{p}$. Furthermore, for $p = \infty$, we define $X ^{\infty}(\mathbb{R}_{+})$ as the space of all measurable functions $f : \mathbb{R}_{+} \mapsto\mathbb{C}$ 
such that $\|f\|_{X ^{\infty}} (\mathbb{R}_{+}):=\sup_{x>0}|f(x)|<\infty$.

In  spaces  $X^{p}(\mathbb{R}_{+})$ we consider the one-parameter $C_{0}$-group of operators 
  $U(t),\>t\in \mathbb{R},$ where
\begin{equation}\label{U}
U(t)f(x)=f(e^{t}x),\>\>\>U(t+\tau)=U(t)U(\tau),
\end{equation}
whose infinitesimal generator is 
\begin{equation}
\frac{d}{dt}U(t)f(x )|_{t=0}=x\frac{d}{dx}f(x)=\Theta f(x).
\end{equation}
The described setting closely related to the Mellin transform 
\begin{equation}\label{mt1}
\mathcal{M}f(s)=\int_{0}^{\infty}f(x)x^{s-1}dx,
\end{equation}
and its inverse
\begin{equation}\label{mt2}
f(x)=\frac{1}{2\pi i}\int_{c-i\infty}^{c+i\infty}\mathcal{M}f(s)x^{-s}ds,\>\>\>s=c+it.
\end{equation}
This pair of transforms has many deep and fascinating connection to many fundamental notions in analysis. For a detailed account of the Mellin Harmonic Analysis, its history, and  connections to other mathematical ideas see the beautiful paper \cite{BJ1}.     A lot of relevant information can also be found in interesting articles \cite{BBO}, \cite{BJ-3}.                    
In particular, in a series of papers  \cite{BBMS-1}-\cite{BBMS-9}, \cite{BJ2}, \cite{Pes22} authors developed in the framework of one-dimensional Mellin analysis analogs of such important topics as Bernstein spaces, Bernstein inequality, Paley-Wiener theorem, Riesz-Boas interpolation formulas, different sampling results including exponential sampling. 

The present paper has two main objectives: 1) to go from one-dimensional setting to multidimensional, and 2)  to concentrate on such important  topics as Besov spaces, $K$-functor, approximations, frames. As well as we know, such topics were never discussed in the setting of the multidimensional Mellin analysis.

The paper is organized as follows. 
In section 2 we describe the basic assumptions and notations of multidimensional Mellin analysis. In section 3 
in the setting of the multidimensional Mellin analysis we introduce moduli of continuity and use them to define Besov-Mellin spaces. In section 4 we prove that Besov-Mellin spaces are the interpolation spaces (in the sense of J.Peetre) between two Sobolev-Mellin spaces.  In section 5 we introduce Bernstein-Mellin spaces and in section 6 prove corresponding direct and inverse approximation theorems including a Jackson-type inequality. Sections 7-9 are devoted to the Hilbert case. We introduce and discuss Laplace-Mellin operator in section 7. In section 8 we define relevant Paley-Wiener spaces and in section 9 consider a partition of unity on the spectrum of the Laplace-Mellin operator. In section {\ref{MoreBes}  we describe Besov-Mellin spaces in terms of approximations by Paley-Wiener functions and in terms of some Hilbert frames.  The article also contains extensions of the Landau-Kolmogorov-Stein inequality to the Mellin analysis setting. Throughout the paper  we extensively use the Peetre's theory of interpolation and approximation spaces. To make the paper self contained some related definitions, facts and proofs are collected in Appendices 1-3. 

All our statements concerning the multidimensional Mellin analysis seems to be new. However, at least some of them are just adaptations and particular realizations of previously obtained (but maybe not very well known)  more general results (see \cite {BB},  \cite{Pes78}-\cite{Pes22}, \cite{T}). In this sense the article can be treated as a review paper. Our main goal was  to attract  attention to multidimensional Mellin analysis, which is currently underdeveloped.

\section{Multidimensional Mellin Analysis}

\subsection{Space $X^{p}=X^{p}(\mathbb{R}_{+}^{n}, d\mu)$. Mellin translations and their infinitesimal operators.}

We set $\mathbb{R}_{+}=\{x\in \mathbb{R},\>x>0\}$ and $\mathbb{R}_{+}^{n}=\{(x_{1}, .., x_{n}), x_{j}\in \mathbb{R}_{+}, \>j=1, .., n\}$, and introduce the measure  $d\mu=  \frac{dx_{1},...dx_{n}}{x_{1}... x_{n}}$. The corresponding space $L^p,\>1\leq p<  \infty$ of complex valued 
functions defined on $\mathbb{R}_{+}^{n}$ is denoted as $X^{p}=X^{p}(\mathbb{R}_{+}^{n}, d\mu)$ and its norm is
$$
\|f\|_{X^{p}}=\left(\int_{\mathbb{R}_{+}^{n}}|f|^{p} d\mu \right)^{1/p},
$$
 for $ 1\leq p<\infty$. We define $X ^{\infty}$ as the space of all measurable functions $f : \mathbb{R}_{+}^{n} \mapsto\mathbb{C}$ 
such that 
$$\|f\|_{X^{\infty}}:=\rm ess \>sup_{x=(x_{1}, ..., x_{n})\in \mathbb{R}_{+}^{n} }|f(x)|<\infty.
$$

In  spaces  $X^{p},\>\>1\leq p<\infty,$ we consider the one-dimensional $C_{0}$-groups of operators 
  $U_{j}(t),\>t\in \mathbb{R},\>j=1, ..., n,$ where
\begin{equation}\label{U}
U_{j}(t)f(x_{1}, .., x_{n})=f(x_{1}, ..., e^{t}x_{j}, ..., x_{n}),
\end{equation}
whose infinitesimal generator is $\>
\Theta_{j}= x_{j}\partial_{j},  \>\>j=1, ..., n,
$
since  
\begin{equation}\label{Dom}
\lim_{t\rightarrow 0}\frac{1}{t}\left(U_{j}(t)f(x_{1}, ..., x_{n})-f(x_{1}, ..., x_{n})\right)=
\frac{d}{dt}U_{j}(t)f(x )|_{t=0}=x_{j}\partial_{j}f(x).
\end{equation}
The general theory (see \cite{BB}, \cite{K}) implies  that  such an operator is closed in $X^{p},\>\>1\leq p<\infty,$. Its domain $\mathcal{D}(\Theta_{j})$ is dense in $X^{p},\>\>1\leq p<\infty,$ and it consists of all functions in $X^{p}$ for which the limit (\ref{Dom}) exists in the  norm of $X^{p},\>\>1\leq p<\infty$.
Domains of the  powers of $\Theta_{j}$ are $\mathcal{D}(\Theta_{j}^{k}),\>\>k\in \mathbb{N}, $  and  $\mathcal{D}^{\infty}(\Theta_{j})=\cap_{k=1}^{\infty}\mathcal{D}(\Theta_{j}^{k})$. We will also use the notation  
$$
\mathcal{D}^{\infty}=\mathcal{D}^{\infty}(\Theta_{1}, ..., \Theta_{n})=\cap_{k\in \mathbb{N}}\cap_{|\alpha|=k} \mathcal{D}\mathcal(\Theta_{1}^{\alpha_{1}} ... \Theta_{n}^{\alpha_{n}}),
$$
where $\alpha=(\alpha_{1}, ..., \alpha_{n}),\>\>\alpha_{j}\in \mathbb{N},\>\>1\leq j\leq n.$

We note that the groups $U_{j}$ commute with each other  
$$
U_{i}(t_{i})U_{j}(t_{j})f=U_{j}(t_{j})U_{i}(t_{i})f,\>\>\>\>f\in X^{p},\>\>1\leq p<\infty,\>\>1\leq i,j\leq n,\>\>t_{i}, t_{j} \geq 0.
$$ 
Also, one can easily verify that $\mathcal{D}^{\infty}$ is invariant under all $U_{j}$ and all $\Theta_{j}$ and
\begin{equation}\label{assumpt}
\Theta_{j}U_{1}(t_{1})U_{2}(t_{2}) ... U_{n}(t_{n})f=U_{1}(t_{1})U_{2}(t_{2}) ... \Theta_{j}U_{j}(t_{j}) ... U_{n}(t_{n}) f,\>\>\>f\in \mathcal{D}^{\infty}.
\end{equation}
One can also check that for the
$$
Uf(t_{1}, t_{2},  ..., t_{n})=U_{1}(t_{1}) U_{2}(t_{2})  ... U_{n}(t_{n})f,\>\>\>(t_{1}, ... , t_{n})\in \mathbb{R}^{n},
$$
the following formula holds
\begin{equation}\label{assumpt-1}
\Theta_{j}Uf(t_{1}, ..., t_{n})=\partial_{j} \left[Uf(t_{1}, ..., t_{n})\right],\>\>\>f\in \mathcal{D}^{\infty}.
\end{equation}

\section{Sobolev-Mellin and Besov-Mellin spaces}\label{S_M_B_M}
\begin{defn}
The Banach space $\mathcal{W}^{k}_{p}=\mathcal{W}^{k}_{p}(\mathbb{R}^{n}_{+}, d\mu),\>\>k\in \mathbb{N},,\>\>1\leq p<\infty,$ which we call the Sobolev-Mellin space,  is the set of functions $f$ in $X^{p}=X^{p}(\mathbb{R}_{+}^{n}, d\mu )$ for which the following norm is finite
$$
\|f\|_{\mathcal{W}^{k}_{p}}=\|f\|_{X^{p}}+\sum_{|\alpha|=k} \|\Theta_{1}^{\alpha_{1}} ... \Theta_{n}^{\alpha_{n}}f\|_{X^{p}},
$$
where $\alpha=(\alpha_{1}, ..., \alpha_{n}),\>|\alpha|=\alpha_{1}+...+\alpha_{n},\>\alpha_{1}, ... ,\alpha_{n}\in \mathbb{N}$.
\end{defn}
By using closeness of the operators $\Theta_{j}$ one can prove the following fact.

\begin{lem}The norm $\|f\|_{\mathcal{W}^{k}_{p}}$ is equivalent to the norm
$$
|||f|||_{\mathcal{W}^{k}_{p}}=\|f\|_{X^{p}}+\sum_{r=1}^{k}\sum_{1\leq j_{1},  ... , j_{r}\leq n}\|\Theta_{j_{1}}...\Theta_{j_{r}}f\|_{X^{p}},
$$
 where $r\in \mathbb{N},\>\> f\in X^{p},\>\>1<p<\infty.$
 \end{lem}
 
One has the following version of the Landau-Kolmogorov-Stein  inequality which follows from a general fact Theorem \ref{KSM-Ap} in Appendix 2.
\begin{thm} The following holds for $f\in \mathcal{D}(\Theta_{j}^{m}),\>\>1\leq p<\infty,$
\begin{equation}
\left\|\Theta_{j}^{k}f\right\|_{X^{p}}^m \leq C_{k,m}\|\Theta_{j}^{m}f\|_{X^{p}}^{k}\|f\|_{X^{p}}^{m-k},\>\>\>
0\leq k \leq m,\label{KSM}
\end{equation}
where $C_{k,m}= (K_{m-k})^m/(K_{m})^{m-k},$ and 
$$
K_{j}=\frac{4}{\pi}\sum_{r=0}^{\infty}\frac{(-1)^{r(j+1)}}{(2r+1)^{j+1}},\>\>\>\>\>\>
j,\>r\in \mathbb{N},
$$
 is a Krein-Favard constant. 
\end{thm}

 The mixed modulus of continuity is introduced as
\begin{equation}
\Omega^{r}_{p}( s, f)=
$$
$$
\sum_{1\leq j_{1},...,j_{r}\leq
n}\sup_{0\leq\tau_{j_{1}}\leq s}...\sup_{0\leq\tau_{j_{r}}\leq
s}\|
\left(U_{j_{1}}(\tau_{j_{1}})-I\right)...\left(U_{j_{r}}(\tau_{j_{r}})-I\right)f\|_{X^{p}},\label{M}
\end{equation}
where $f\in X^{p},\ r\in \mathbb{N},  $ and $I$ is the
identity operator in $X^{p},\>\>1\leq p<\infty.$

\begin{defn}
For $\alpha\in \mathbb{R}_{+} ,\>\>1\leq p<\infty, \>\>1\leq q\leq \infty,$ and $\alpha$ is not integer the Besov-Mellin space $\mathcal{B}^{\alpha}_{p,q}$ is defined as the subspace in $X^{p}$ of all the functions for which the following norm is finite
\begin{equation}\label{Bnorm1}
\|f\|_{\mathcal{W}^{[\alpha]}_{p}}+\sum_{1\leq j_{1},...,j_{[\alpha] }\leq n}
\left(\int_{0}^{\infty}\left(s^{[\alpha]-\alpha}\Omega^{1}_{p}
(s,\Theta_{j_{1}}... \Theta_{j_{[\alpha]}}f)\right)^{q}\frac{ds}{s}\right)^{1/q},
\end{equation}
where $[\alpha]$ is the integer part of $\alpha$.

In the case when 
$\alpha=k\in \mathbb{N}$ is an integer  the Besov-Mellin space $\mathcal{B}^{\alpha}_{p,q},\>\>1\leq p<\infty, \>\>1\leq q\leq \infty,$ is defined as the subspace in $X^{p}$ of all the functions for which the following norm is finite
 (Zygmund condition)
\begin{equation}\label{Bnorm2}
\|f\|_{\mathcal{W}^{k-1}_{p}}+ \sum_{1\leq j_{1}, ... ,j_{k-1}\leq n }
\left(\int_{0}^{\infty}\left(s^{-1}\Omega^{2}_{p}(s,
\Theta_{j_{1}}...\Theta_{j_{k-1}}f)\right)
 ^{q}\frac{ds}{s}\right)^{1/q}.
\end{equation}
\end{defn}

\section{The interpolation theorem for the Sobolev-Mellin  and Besov-Mellin spaces}\label{Interp}

It is well known that the most important  property  of the  classical  Besov spaces  is the fact that they are interpolation spaces (see section Appendix 1 below) between two Sobolev spaces \cite{BB}, \cite{KPS}. Here we formulate this result for Besov-Mellin spaces.

For the pair of Banach spaces $(X^{p}, \mathcal{W}_{p}^{r}),,\>\>1\leq p<\infty, \>\>$  the $K$-functor is defined by the formula 
$$
K(s^{r}, f,  X^{p}, \mathcal{W}^{r}_{p})=\inf_{f=f_{0}+f_{1},f_{0}\in X^{p},
f_{1}\in \mathcal{W}_{p}^{r}}\left(\|f_{0}\|_{X^{p}}+s^{r}\|f_{1}\|_{\mathcal{W}_{p}^{r}}\right).\label{K}
$$

The results below  are the particular case of the results in \cite{Pes79}, \cite{Pes19} (see also \cite{KP}, \cite{Pes78}-\cite{Pes83}, \cite{Pes88b}) which correspond to a situation with commuting one-parameter semigroups.

We consider  the so-called Hardy-Steklov-Mellin  operator $\mathcal{P}_{r}(s)$  which is defined on $X^{p},\>\>1\leq p<\infty,$ by the formula 
$$
\mathcal{P}_{r}(s)=\mathcal{P}_{1,r}(s)\mathcal{P}_{2,r}(s) ... \mathcal{P}_{n,r}(s)f,\>\>\>f\in X^{p}, \>\>1\leq p<\infty,
$$
where
$$
\mathcal{P}_{j.r}(s)=(s/r)^{-r}\int_{0}^{s/r}...\int_{0}^{s/r}\sum
_{k=1}^{r}(-1)^{k}
C^{k}_{r}U_{j}\left( k(\tau_{j,1}+...+\tau_{j,r})\right)fd\tau_{j,1}...d\tau_{j,r},
$$
and  $C^{k}_{r}$ are the binomial coefficients. 
As it is  shown in   \cite{Pes79}, \cite{Pes19}, 
$$
\mathcal{P}_{r}(s): X^{p} \mapsto \mathcal{W}_{p}^{r},\>\>1\leq p<\infty,\>r\in \mathbb{N}, \>s\in \mathbb{R}.
$$
Moreover,  there exist constants $c>0,\>C>0$, such that for all $f\in X^{p},\>\>s\geq 0, \>\>1\leq p<\infty, \>\>$
\begin{equation}\label{main-ineq}
c \>\Omega^{r}_{p}(s, f)\leq K(s^{r}, f, X^{p}, \>\mathcal{W}^{r}_{p})\leq C \left( \Omega^{r}_{p}(s, f)+\min(s^{r}, 1)\|f\|_{X^{p}} \right).
\end{equation}
This inequality implies density of $\mathcal{W}^{r}_{p},\>\>1\leq p<\infty, \>$ in $X^{p}$.

\begin{defn}The interpolation space $\left(X^{p}, \>\mathcal{W}^{r}_{p}\right)^{K}_{\alpha/r, q},             \ 0<\alpha<r\in \mathbb{N},\ 1\leq p<\infty,\>1\leq  q\leq \infty,$ is defined as the space of all $f\in X^{p}$ for which the following norm is finite
\begin{equation}\label{B-Int}
\|f\|_{\alpha/r, q}^{K}=\|f\|_{X^{p}}+\left (\int_{0}^{\infty}\left(s^{\alpha}   K(s^{r}, f, X^{p}, \>\mathcal{W}^{r}_{p})\right)^{q}\frac{ds}{s}                \right)^{1/q}, \>\>\> 1\leq q<\infty,
\end{equation}
with the usual modifications for $q=\infty$.
\end{defn}

\begin{thm}\label{Main} The following holds true.

\begin{enumerate}

\item
 The Besov-Mellin  space $\mathcal{B}^{\alpha}_{p q}$ coincides with the interpolation space $\left(X^{p}, \>\mathcal{W}^{r}_{p}\right)^{K}_{\alpha/r,q},             \ 0<\alpha<r\in \mathbb{N},\ 1\leq p<\infty,\>1\leq  q\leq \infty,$ and the norms (\ref{Bnorm1}), (\ref{Bnorm2}), (\ref{B-Int}) are equivalent (for appropriate choice of indices) to the  norm
\begin{equation}\label{Bnorm3}
\|f\|_{X^{p}}+\left(\int_{0}^{\infty}(s^{-\alpha}\Omega^{r}_{p}(s,
f))^{q} \frac{ds}{s}\right)^{1/q} , \>\>\>1\leq q<\infty,
\end{equation}
with the usual modifications for $q=\infty$.

\item The following isomorphism holds true
$$
\left(X^{p}, \>\mathcal{W}^{r}_{p}\right)^{K}_{\alpha/r,q}=\left(\mathcal{W}^{k_{1}}_{p},\>\mathcal{W}^{k_{2}}_{p}\right)^{K}_{(\alpha-k_{1})/(k_{2}-k_{1}),q},
$$
where $0\leq k_{1}<\alpha<k_{2}\leq r\in \mathbb{N},\> ,\>\>1\leq p<\infty, \>\>1\leq  q\leq \infty.$

\end{enumerate}

\end{thm}

 \bigskip

\begin{rem}
The results of the last two sections didn't use the fact that the operators $\Theta_{j}$ commute with each other. However, by using commutativity of $\Theta_j$ one can show (see \cite{T}) that the norm of $\mathcal{W}_{p}^{r}$ is equivalent to the norm
$$
\|f\|_{X^{p}}+\sum_{j=1}^{n}\|\Theta_{j}f\|_{X^{p}},
$$
and the norm (\ref{Bnorm3}) is equivalent to the norm
$$
\|f\|_{X^{p}}+\sum_{j=1}^{n}\left( \int_{0}^{\infty}\left(s^{-\alpha}w_{j,p}^{r}\left(s,f\right)\right)^{q}\right)^{1/q}, \>\>1\leq q<\infty,
$$
where 
$$
\omega_{j,p}^{m}\left( s, f\right)=\sup_{0\leq \tau\leq s}\|\left(U_{j}(\tau)-I\right)^{m}f\|_{X^{p}}.
$$
\end{rem}

\section{ Bernstein-Mellin spaces}

Let's remind that in the classical analysis a  Bernstein class \cite{A}, \cite{N}, which is denoted as ${\bf B}_{\sigma}^{p}(\mathbb{R}),\>\> \sigma\geq 0, \>\>1\leq p\leq  \infty,$ is a linear space of all functions $f:\mathbb{R} \mapsto \mathbb{C}$ which belong to $L^{p}(\mathbb{R})$ and admit  extension to $\mathbb{C}$ as entire functions of exponential type $\sigma$. It is known,  that for every such function    there exists a constant $A>0$ such the following inequality holds
$$
|f(x+iy)|\leq Ae^{\sigma|y|}.
$$
A function $f$ belongs to ${\bf B}_{\sigma}^{p}(\mathbb{R})$ if and only if the following Bernstein inequality holds 
$$
\|f^{(k)}\|_{L^{p}(\mathbb{R})}\leq \sigma^{k}\|f\|_{L^{p}(\mathbb{R})}
$$
for all natural $k$. Using the distributional Fourier transform 
$$
\widehat{f}(\xi)=\frac{1}{\sqrt{2\pi}}  \int_{\mathbb{R}} f(x)e^{-i\xi x}dx, \>\>\>f\in L^{p}(\mathbb{R}), \>\>1\leq p\leq  \infty,
$$
one can show (Paley-Wiener theorem) that $f\in {\bf B}_{\sigma}^{p}(\mathbb{R}), \>\>1\leq p\leq \infty,$ if and only if $f\in L^{p}(\mathbb{R}),\>\>1\leq p \leq\infty,$ and the support of $\widehat{f}$ (in sens of distributions) is in $[-\sigma, \sigma]$.

\begin{defn}
The Bernstein-Mellin space 
 ${\bf B}_{ \sigma}^{p}(\Theta_{1}, ..., \Theta_{n}), \>\>\sigma>0,\>1\leq p< \infty,$ is defined as  a set of all functions $f$ in $X^{p}$ 
 which belong to $\mathcal{D}^{\infty}(\Theta_{1}, ..., \Theta_{n})$  and for which 
\begin{equation}\label{Bernstein1}
\|\Theta_{j_{1}}^{l_{1}} ... \Theta_{j_{k}}^{l_{k}}f\|_{X^{p}}\leq \sigma^{|l|}\|f\|_{X^{p}}, 
\end{equation}
where $ \>l=(l_{j_{1}}, ..., l_{j_{k}}),\>|l|=l_{j_{1}}+...+l_{j_{k}},\>l_{j}\in \mathbb{N}\cup\{0\},\>k\in \mathbb{N}.$

The notation ${\bf B}_{ \sigma}^{p}(\Theta_{j}),\>1\leq p< \infty,$ will be used for the set of all functions $f$ in $X^{p}$ 
 which belong to $\mathcal{D}^{\infty}(\Theta_{j})$  and for which 
\begin{equation}\label{Bernstein}
\|\Theta_{j}^{l}f\|_{X^{p}}\leq \sigma^{l}\|f\|_{X^p}.
\end{equation}
\end{defn}

The next Lemma follows from a more general result Lemma \ref{BernstCriterion-Ap} in Appendix 2.

\begin{lem}\label{BernstCriterion}

A function $f\in \mathcal{D}^{\infty}$ belongs to  ${\bf B}_{\sigma}^{p}(\Theta_{j}), \>1\leq j\leq n,>\>\sigma> 0 ,\>1\leq p< \infty,$ if and only if
 the quantity
\begin{equation}\label{Lem}
\sup_{k\in N} \sigma ^{-k}\|\Theta_{j}^{k}f\|_{X^{p}(\mathbb{R}_{+})}=R(f,\sigma)
\end{equation}
 is finite.

\end{lem}

\begin{thm}\label{MultiBernstCriterion}

A function $f\in \mathcal{D}^{\infty}$ belongs to  ${\bf B}_{\sigma}^{p}(\Theta_{1}, ...,\Theta_{n}), \>\>\sigma> 0 ,\>1\leq p< \infty,$ if and only if
 the quantity
\begin{equation}\label{Lem-0}
\sup_{|l|\in \mathbb{N}} \>\>\sup_{l=(l_{j_{1}}, ..., l_{j_{k}})}\sigma ^{-|l|}\|\Theta_{j_{1}}^{l_{j_{1}}} ... \Theta_{j_{k}}^{l_{j_{k}}}f\|_{X^{p}}=R(f,\sigma),\>\>\>|l|=l_{j_{1}}+ ... +l_{j_{k}},
\end{equation}
 is finite.
\end{thm}
\begin{proof} It is clear that if $f\in {\bf B}_{\sigma}^{p}(\Theta_{1}, ... ,\Theta_{n})$ then (\ref{Lem-0}) holds.
 Now, we assume that (\ref{Lem-0}) holds and we have to verify the inequality (\ref{Bernstein1}) with $k=n$. The inequality 
(\ref{Lem-0}) and the Lemma \ref{BernstCriterion} imply that for every $1\leq j\leq n$ the function $f$ belongs to $ {\bf B}_{\sigma}^{p}(\Theta_{j})$, i.e.
$$
\|\Theta_{j}^{\alpha}f\|_{X^{p}}\leq \sigma^{\alpha}\|f\|_{X^{p}},\>\alpha\in \mathbb{N}.
$$
 
 We proceed by induction assuming  that for some $1\leq k<n$ the inequality  (\ref{Lem-0}) implies (\ref{Bernstein}). Clearly, (\ref{Lem-0})  implies that any $\Theta_{j_{1}}^{\alpha_{1}} ... \Theta_{j_{k}}^{\alpha_{k}}f$  satisfies (\ref{Bernstein}) for any $\Theta_{j}$ and Lemma \ref{BernstCriterion} implies the inequality
 $$
 \|\Theta^{\alpha_{j_{k+1}}}_{j_{k+1}}\Theta_{j_{1}}^{\alpha_{1}} ... \Theta_{j_{k}}^{\alpha_{k}}f\|_{X^{p}}\leq \sigma^{|\alpha_{j_{k+1}}|}\|\Theta_{j_{1}}^{\alpha_{1}} ... \Theta_{j_{k}}^{\alpha_{k}}f\|_{X^{p}}\leq \sigma^{ \left( |\alpha_{j_{k+1}}|+ ... + |\alpha_{j_{k}}|\right)}\|f\|_{X^{p}}.
 $$
 
Theorem is proven.

\end{proof}

For a function $f\in \cup_{\sigma>0}\mathbf{B}_{\sigma}^{p}(\Theta_{j}),$ the $\sigma_{j ; f}$ will denote the smallest $\sigma$ for which the inequality ({\ref{Bernstein}) holds.

The next statement follows from the corresponding Theorem \ref{new-Ap}  in our Appendix 2.

\begin{thm}\label{new}
If a function $f$ belongs   to the space  $ \cup_{\sigma>0}\mathbf{B}_{\sigma}^{p}(\Theta_{j}),  $ then the following limit exists
 \begin{equation}
 d_{j; f}=\lim_{k\rightarrow \infty} \|\Theta_{j}^k
f\|^{1/k}_{X^{p}}, \label{limit}
\end{equation}
and $d_{j,f}=\sigma_{j,f}.$ 
Conversely, if
$f\in \mathcal{D}^{\infty}(\Theta_{j})$ and $d_{j,f}=\lim_{k\rightarrow \infty}
\|\Theta_{j}^k f\|^{1/k}_{X^{p}},$ exists and is finite, then $f\in\mathbf{B}^{p}_{d_{j,f}}(\Theta_{j})$
and $d_{j,f}=\sigma_{j,f }.$
\end{thm}

\section{Direct and inverse approximation theorems  in Mellin analysis. A Jackson-type inequality.}

The lemma below is a particular case of the Lemma \ref{BernstConstruction-Ap} in Appendix 2.

\begin{lem}\label{BernstConstruction}\cite{Pes15}
Suppose that  $\mu\in {\bf B}^{1}_{\sigma}(\mathbb{R})$ is an entire function of exponential
type $\sigma$.  Then for any $f\in X^{p}$ the function
$$
P_{j, \sigma}f=\int _{-\infty}^{\infty}\mu(t)U_{j}(t)fdt
$$
belongs to ${\bf B}_{\sigma}^{p}(\Theta_{j}).$  The integral here is understood in the sense of Bochner for the abstract-valued function $\mu(t)U_{j}(t)f: \mathbb{R}\mapsto X^{p},\>1\leq p< \infty$.

\end{lem}

\begin{thm}\label{MultiBernstConstruction}

Let  $\mu\in L_{1}(\mathbb{R}),\>\>\|\mu\|_{L_{1}(\mathbb{R})}=1,$ be an entire function of exponential
type $\sigma$ then for any $f\in X^{p},\>1\leq p< \infty,$ the function
$$
P_{\sigma}f=\underbrace{\int _{-\infty}^{\infty} ... \int_{-\infty}^{\infty}}_{n}\mu(\tau_{1}) ...
\mu(\tau_{n})U_{1}(\tau_{1}) ... U_{n}(\tau_{n}) fd\tau_{1} .. d\tau_{n}
$$
belongs to ${\bf B}_{\sigma}^{p}(\Theta_{1}, ... , \Theta_{n} ),\>1\leq p< \infty$.  The integral here is understood in the sense of Bochner for the abstract-valued function $\mu(\tau_{1}) ...
\mu(\tau_{n})U_{1}(\tau_{1}) ... U_{n}(\tau_{n}) f: \mathbb{R}^{n}\mapsto X^{p}$.
\end{thm}
\begin{proof}
Indeed, due to commutativity and the previous Lemma one has
$$
\|\Theta_{j_{1}}^{l_{1}} \Theta_{j_{2}}^{l_{2}} ... \Theta_{j_{k}}^{l_{k}}P_{\sigma}f\|_{X^{p}}=
\left\|\Theta_{j_{1}}^{l_{1}}\int_{-\infty}^{\infty}\mu(\tau_{j_{1}}) U_{j_{1}} (\tau_{j_{1}} ) \left[T_{j_{1}}(f)\right] d\tau_{j_{1}}  \right \|_{X^{p}}\leq \sigma^{l_{1}}\|T_{j_{1}}(f)\|_{X^{p}},
$$
where 
$$
T_{j_{1}}(f)=
$$
$$
\Theta_{j_{2}}^{l_{2}} ... \Theta_{j_{k}}^{l_{k}}\underbrace{\int _{-\infty}^{\infty} ... \int_{-\infty}^{\infty}}_{n-1}\mu(\tau_{1}) ...\widehat{\mu(\tau_{j_{1}})}...
\mu(\tau_{n})U_{1}(\tau_{1}) ...\widehat{U_{j_{1}}(\tau_{j_{1}})}... U_{n}(\tau_{n}) fd\tau_{1} .. . \widehat{d\tau_{j_{1}}} ... d\tau_{n},
$$
where a wedge means that the corresponding term is missing.
It is obvious that continuing this way we obtain
$$
\|\Theta_{j_{1}}^{l_{1}} \Theta_{j_{2}}^{l_{2}} ... \Theta_{j_{k}}^{l_{k}}P_{\sigma}f\|_{X^{p}}\leq \sigma^{|l|}\|f\|_{X^{p}}.
$$
Applying Theorem \ref{MultiBernstCriterion} we finish the proof.
\end{proof}

Let's consider the functionals

$$
\omega_{j,p}^{m}\left( s, f\right)=\sup_{0\leq \tau\leq s}\|\left(U_{j}(\tau)-I\right)^{m}f\|_{X^{p}},
$$
$$
\mathcal{E}_{j,p}(\sigma, f)=\inf_{g\in {\bf B}_{\sigma}^{p}(\Theta_{j})}\|f-g\|_{X^{p}},
$$
where $1\leq p<\infty$.
One has  the following inequalities:
\begin{equation}\label{first}
\omega_{j,p}^{m}\left( s, f\right)\leq s^{k}\omega^{m-k}_{j,p}(s, \Theta_{j}^{k}f),\>\>\>\>0\leq k\leq m,
\end{equation}
and
\begin{equation}\label{second}
\omega_{j,p}^{m}\left(as, f\right)\leq \left(1+a\right)^{m}\omega_{j,p}^{m}(s, f), \>\>\>a\in \mathbb{R}_{+}.
\end{equation}
The first one follows from the identity
\begin{equation}
\left(U_{j}(s)-I\right)^{k}f=\int_{0}^{s}...\int_{0}^{s}U_{j}(\tau_{1}+... +\tau_{k})
\Theta_{j}^{k}fd\tau_{1}...d\tau_{k},\label{id3}
\end{equation}
where $I$ is the identity operator and $k\in \mathbb{N}$.  The
second one follows from the property
$$
\omega_{j,p}^{1}\left( s_{1}+s_{2}, f\right)\leq\omega_{j,p}^{1}\left(
s_{1}, f\right)+\omega_{j,p}^{1}\left(s_{2}, f\right)
$$
which is easy to verify. 

Our next objective is to prove an analog of the Jackson theorem.
The proof of the lemma below  is  motivated by  Theorem 5.2.1  in \cite{N}.

Let
\begin{equation}
\rho(t)=a\left(\frac{\sin (t/N)}{t}\right)^{N}
\end{equation}
where $N=2( m+3)$  and
$$
a=\left(\int_{-\infty}^{\infty}\left(\frac{\sin
(t/N)}{t}\right)^{N}dt\right)^{-1}.
$$
With such choice of $a$ and $N$ function $\rho$ will have the
 following properties:

(1) $\rho$ is an even nonnegative entire function of exponential
type one;

(2) $\rho$ belongs to $L_{1}(\mathbb{R})$ and its
$L_{1}(\mathbb{R})$-norm is $1$;

(3)  the integral
\begin{equation}
\int_{-\infty}^{\infty}\rho(t)|t|^{m}dt
\end{equation} 
is finite.

Next, we observe the following formula for every $1\leq j\leq n$
\begin{equation}\label{binom}
(-1)^{m+1}(U_{j}(s)-I)^{m}f=
$$
$$
(-1)^{m+1}\sum^{m}_{k=0}(-1)^{m-k}C^{k}_{m}U_{j}(ks)f=
 \sum_{k=1}^{m}b_{k}U_{j}(ks)f-f,
 \end{equation}
 where $b_{1}+b_{2}+ ... +b_{m}=1.$
Consider
 the vector 
\begin{equation}\label{id}
\mathcal{Q}_{j}(\sigma,m)(f)=\int_{-\infty}^{\infty}
\rho(t)\left\{(-1)^{m+1}\left(U_{j}\left(\frac{t}{\sigma}\right)-I\right)^{m}f+f\right\}dt.
 \end{equation}
 According to (\ref{binom}) we have 
 $$
\mathcal{Q}_{j}( \sigma,m)(f) =                  \int_{-\infty}^{\infty}
\rho(t)\sum_{k=1}^{m}b_{k}U_{j}\left(k\frac{t}{\sigma}\right)fdt.
$$
Changing variables in each of integrals 
$$
\int_{-\infty}^{\infty}\rho(t)U_{j}\left(k\frac{t}{\sigma}\right)fdt,\>\>\>\>1\leq k\leq m,
$$
we obtain the formula
$$
\mathcal{Q}_{j}( \sigma,m)(f) =\int_{-\infty}^{\infty}\Phi(t)U_{j}(t)fdt,
$$
where
$$
\Phi(t)=\sum_{k=1}^{m}b_{k}\left(\frac{\sigma}{k}\right)\rho\left(t\frac{\sigma}{k}\right), \>\>\>\>\>\>b_{1}+b_{2}+...
+b_{m}=1.
 $$

Since the function $\rho(t)$ has exponential type one every function
$\rho\left(t\frac{\sigma}{k}\right)$ has the type $\sigma/k,\>\>1\leq k\leq m,$ and because of  this 
the function $\Phi(t)$ 
has exponential  type $\sigma$. It  also belongs to $L_{1}(\mathbb{R})$ and as it was just shown in the previous statement it implies that the function $\mathcal{Q}_{j}( \sigma,m)(f) $ belongs to $\mathbf{B}^{p}_{\sigma}(\Theta_{j})$.

\begin{lem}\label{j-1-group} \cite{Pes15} For a given natural $m$ there exists a constant $c=c(m)>0$ such that for all
$\sigma>0$ and all $f$ in $X^{p},1\leq p< \infty,$ 
\begin{equation}\label{100}
\mathcal{E}_{j,p}(\sigma, f)\leq
c\omega^{m}_{j,p}\left( 1/\sigma, f\right).
\end{equation}
Moreover, for any $1\leq k\leq m$ there exists a $C=C(m,k)>0$ such that for any  $f\in \mathcal{D}(\Theta_{j}^{k})$ one has
\begin{equation}
\mathcal{E}_{j,p}(\sigma, f)\leq
\frac{C}{\sigma^{k}}\omega^{m-k}_{j,p}\left( 1/\sigma, \>\Theta_{j}^{k}f\right),\>\>\>\>
0\leq k\leq m.\label{200}
\end{equation}
\end{lem}
\begin{proof}
We estimate the error of approximation of
$\mathcal{Q}_{j}(\sigma,m)(f) $  to $f$. 
 Since by (\ref{id})
$$
\mathcal{Q}_{j}(\sigma,m)(f) -f=
(-1)^{m+1}\int_{-\infty}^{\infty}\rho(t)\left(U_{j}\left(\frac{t}{\sigma}\right)-I\right)^{m}fdt
$$
we obtain by using (\ref{second}) 
\begin{equation}\label{single-j}
\mathcal{E}_{j,p}(\sigma, f)\leq\left\|f-\mathcal{Q}_{j}( \sigma,m)(f) \right\|_{X^{p}}\leq
\int_{-\infty}^{\infty}\rho(t)\left    \|         \left(U_{j}\left(\frac{t}{\sigma}\right)-I\right)^{m}f  \right\|_{X^{p}}     dt\leq
$$
$$
\int_{-\infty}^{\infty}\rho(t)\omega^{m}_{j,p}\left( t/\sigma,\>f\right)dt \leq c\omega^{m}_{j,p}\left(1/\sigma,\>f\right), \>\>\>\>\>\>\>c=\int_{-\infty}^{\infty}\rho(t)(1+|t|)^{m}dt.
\end{equation}

If $f\in \mathcal{D}(\Theta_{j}^{k})$ then by using (\ref{first}) 
we have
\begin{equation}\label{single-J}
\mathcal{E}_{j,p}(\sigma, f)\leq
\int_{-\infty}^{\infty}\rho(t)\omega^{m}_{j,p}\left(t/\sigma, f\right)dt
\leq 
$$
$$
\frac{\omega^{m-k}_{j,p}\left(
1/\sigma, \>\Theta_{j}^{k}f\right)}{\sigma^{k}}\int_{-\infty}^{\infty}\rho(t)|t|^{k}(1+|t|)^{m-k}dt\leq
\frac{{C}}{\sigma^{k}}\omega^{m-k}_{j,p}\left(
1/\omega, \>\Theta_{j}^{k}f\right),
\end{equation}
where
 $$
C=\int_{-\infty}^{\infty}\rho(t)|t|^{k}(1+|t|)^{m-k}dt,
$$
 is finite by the choice of $\rho$. The inequalities (\ref{100}) and (\ref{200}) are proved.

\end{proof}

We define the following functional
\begin{equation}\label{BestAp}
\mathcal{E}_{p}(\sigma, f)=\inf_{g\in {\bf B}^{p}_{\sigma}(\Theta_{1}, ..., \Theta_{n})}\|f-g\|_{X^{p}}, \>\>1\leq p<\infty.
\end{equation}
Since  the operator $\mathcal{Q}_{j}(\sigma,m),\>1\leq j\leq n,$ maps $X^{p}, \>\>1\leq p<\infty,$ into ${\bf B}_{\sigma}^{p}(\Theta_{j}), \>\>1\leq p<\infty,$ 
and since all the operators $\Theta_{j}, \>U_{j},\>1\leq j\leq n,$ commute with each other one concludes that for any $f\in X^{p}$ the function
$$
\mathcal{Q}_{1}(\sigma,m)... \mathcal{Q}_{n}( \sigma,m)(f) =
$$
$$
\int_{-\infty}^{\infty} ... \int_{-\infty}^{\infty}\Phi(\tau_{1}) ... \Phi(\tau_{n})U_{j}(\tau_{1}+ ... +\tau_{n})fd\tau_{1} ... d\tau_{n},
$$
belongs to ${\bf B}_{\sigma}^{p}(\Theta_{1}, ..., \Theta_{n})$.

Now we can formulate the following analog of the Jackson inequality.
\begin{thm}
There exists a constant $C$ such that for the same notations as above the next inequality holds
\begin{equation}\label{ap-mod}
\mathcal{E}_{p}(\sigma, f)\leq C\sum_{j=1}^{n}\omega^{m}_{j,p}\left(1/\sigma,\>f\right)\leq C\Omega_{p}^{m}(1/\sigma, f).
\end{equation}
\end{thm}
\begin{proof}
By using (\ref{single-j}), the formal identity 
$$
1-a_{1}a_{2}...a_{n}=(1-a_{1})+a_{1}(1-a_{2})+...+ a_{1}a_{2}...a_{n-1}(1-a_{n}),
$$
and boundness of every operator $\mathcal{Q}_{j}( \sigma,m)$, we obtain the following  
\begin{equation}
\mathcal{E}_{p}(\sigma, f)\leq \|f-\mathcal{Q}_{1}( \sigma,m)\mathcal{Q}_{2}( \sigma,m)...\mathcal{Q}_{n}( \sigma,m)f\|_{X^{p}}\leq 
$$
$$
C\sum_{j=1}^{n}\|f-\mathcal{Q}_{j}( \sigma, m)f\|_{X^{p}}\leq 
C\sum_{j=1}^{n}\omega^{m}_{j,p}\left(1/\sigma,\>f\right)\leq C\Omega_{p}^{m}(1/\sigma, f).
\end{equation}
The inequality (\ref{ap-mod}) is proved.
\end{proof}
Set
$$
B=\cup_{\sigma>0}{\bf B}_{\sigma}^{p}(\Theta_{1}, ..., \Theta_{n}), \>\>1\leq p<\infty.
$$
According to the general theory (see Appendix 1) the space
of all $f\in X^{p}$ for which the integral 
$$
\left(\int_{0}^{\infty}\left(\tau^{\alpha}\mathcal{E}_{p}(\tau,f)\right)^{q}\frac{d\tau}{\tau}\right)^{1/q} , \>\>1\leq p<\infty,\>\>1\leq q\leq \infty,
$$
is finite is known as the \textit{approximation space} and denoted by $\mathcal{E}_{\alpha, \>q}(X^{p}, B)$.
We have the following Direct Approximation Theorem.
\begin{thm}\label{Direct-Direct} The following continuous embedding holds $, \>\>1\leq p<\infty,\>\>1\leq q\leq \infty,$
\begin{equation}\label{d-d}
\mathcal{B}_{p, q}^{\alpha}\subset \mathcal{E}_{\alpha, \>q}(X^{p}, B),
\end{equation}
and  for some $C>0$ the inequalities hold
$$
\int_{0}^{\infty}\left(\tau^{\alpha}\mathcal{E}_{p}(\tau,f)\right)^{q}\frac{d\tau}{\tau} \leq C\sum_{j=1}^{n}\int_{0}^{\infty}(s^{-\alpha}\omega^{r}_{j, p}(s,
f))^{q} \frac{ds}{s} \leq                               C\int_{0}^{\infty}(s^{-\alpha}\Omega^{r}_{p}(s,
f))^{q} \frac{ds}{s}.
$$

\end{thm}
\begin{proof}
According to  (\ref{ap-mod})  we have
$$
s^{-\alpha}\mathcal{E}_{p}(s^{-1},f)\leq C s^{-\alpha}\Omega_{p}^{r}(s,f).
$$
Since 
$$
\int_{0}^{\infty}(s^{-\alpha}\mathcal{E}_{p}(s^{-1}, f))^{q} \frac{ds}{s}   =   \int_{0}^{\infty}\left(\tau^{\alpha}\mathcal{E}_{p}(\tau,f)\right)^{q}\frac{d\tau}{\tau},
$$
and since the norm of the Besov space $\mathcal{B}^{\alpha}_{p, q}$ 
 is equivalent to the  norm
\begin{equation}\label{Bnorm300}
\|f\|_{X^{p}}+\left(\int_{0}^{\infty}(s^{-\alpha}\Omega^{r}_{p}(s,
f))^{q} \frac{ds}{s}\right)^{1/q} , \>\>\>\alpha<r,
\end{equation}
our claim  is proven.
\end{proof}

In order to obtain an Inverse Approximation theorem, we are using Theorem \ref{Inverse} with ${\bf E}=X^{p},\>{\bf F}=\mathcal{W}_{p}^{m},\>\mathcal{T}=B$. On the linear space $B=\cup_{\sigma>0}{\bf B}_{\sigma}^{p}(\Theta_{1}, ..., \Theta_{n})$ we consider the quasi-norm defined as
$$
 | f |_{B} = \inf \left \{ \sigma>0~: f \in {\bf B}_{\sigma}^{p}(\Theta_{1}, ..., \Theta_{n}) \right\}~.
$$
Thus the assumption of Theorem \ref{Inverse} simply means that the Bernstein-Mellin inequality for functions in $B$ holds true, i.e.  if $f$ in $B$ belongs to ${\bf B}^{p}_{\sigma}(\Theta_{1}, ... ,\Theta_{n})$ then
$$
\|f\|_{\mathcal{W}_{p}^{m}}\leq \sigma^{m}\|f\|_{X^{p}}.
$$
It implies the Inverse Approximation Theorem meaning that the following continuous imbedding holds true
$$
\mathcal{E}_{\alpha, \>q}(X^{p}, B)\subset \mathcal{B}_{p, q}^{\alpha}, \>\>1\leq p<\infty,\>\>1\leq q\leq \infty, \>\>\alpha\in \mathbb{R}_{+}.
$$

\begin{rem}
We obtain the Direct Approximation Theorem \ref{Direct-Direct} directly as a consequence of the Jackson-type inequality (\ref{ap-mod}). However, one could also use the abstract Direct Theorem \ref{Direct} to obtain Theorem \ref{Direct-Direct}. Indeed, first we note that 
 by using (\ref{single-J}),  (\ref{ap-mod}) and (\ref{first}), (\ref{second}) one has
$$
  \mathcal{E}_{p}(\sigma, f)\leq  C\sum_{j=1}^{n}\|f-\mathcal{Q}_{j}( \sigma, m)f\|_{X^{p}}\leq 
C_{1}\sum_{j=1}^{n}\int_{-\infty}^{\infty}\rho(t)\omega^{m}_{j,p}\left(t/\sigma,\>f\right)dt\leq 
$$
$$
C_{1}\sum_{j=1}^{n}\frac{\omega^{m-k}_{j,p}\left(
1/\sigma, \>\Theta_{j}^{k}f\right)}{\sigma^{k}}\int_{-\infty}^{\infty}\rho(t)|t|^{k}(1+|t|)^{m-k}dt\leq
C_{2}\sum_{j=1}^{n}\frac{{1}}{\sigma^{k}}\omega^{m-k}_{j,p}\left(
1/\omega, \>\Theta_{j}^{k}f\right),
$$
and for $m=k$ the last inequality implies that
\begin{equation}\label{LimitingJackson}
 \mathcal{E}_{p}(\sigma, f)\leq C_{2}\sigma^{-m}\|f\|_{\mathcal{W}_{p}^{m}},\>\>\>f\in \mathcal{W}_{p}^{m}.
\end{equation}
Comparing our situation with Theorem \ref{Direct} we see that ${\bf E}=X^{p},\>{\bf F}=\mathcal{W}_{p}^{m},\>\mathcal{T}=B$, and the assumption of this theorem is exactly (\ref{LimitingJackson}) with the $\beta=m$. All together  it gives the imbedding  (\ref{d-d}).
\end{rem}

\section{The Laplace-Mellin operator in the Hilbert space $X^2$}

Since $\Theta_{j}=x_{j}\partial_{j}$ is a skew-symmetric operator in $X^{2}$ its 
 negative square 
$$
-\Theta_{j}^{2}=-\left(x_{j}\partial_{j}\right )^{2}=-x^{2}_{j}\partial_{j}^{2}-x_{j}\partial_{j}
$$ 
is a self-adjoint non-negative operator. We will use notation $e^{-it\Theta_{j}^{2}}$ for the corresponding group of unitary operators.

\begin{rem}
We note that the opertator $
-\Theta_{j}^{2}=-\left(x_{j}\partial_{j}\right )^{2}=-x^{2}_{j}\partial_{j}^{2}-x_{j}\partial_{j}
$ is a "multiple"of the Bessel operator  
$$
\mathcal{B}_{j}=-\partial_{j}^{2}-\frac{1}{x_{j}}\partial_{j}
$$
in the sense that $\Theta_{j}^{2}=x_{j}^{2}\mathcal{B}_{j}$.
\end{rem}

In the Hilbert space $\mathcal{W}^{k}_{2}$ we introduce the Laplace-Mellin operator as 
\begin{equation}\label{L}
L_{\mathcal{M}}=-\sum _{j=1}^{n}\Theta_{j}^{2}=-\sum_{j=1}^{n}\left(x^{2}_{j}\partial_{j}^{2}+x_{j}\partial_{j}    \right).
\end{equation}

 The self-adjoint operator $L_{\mathcal{M}}$ generates a group of unitary operators $e^{-itL_{\mathcal{M}}}$
and since the operators $\Theta_{j}^{2},\>j=1,2, ..., n,$ commute with each other,
the following formula holds
$$
e^{-itL_{\mathcal{M}}}=e^{-it\Theta_{1}^{2}}e^{-it\Theta_{2}^{2}} ... e^{-it\Theta_{n}^{2}}.
$$
The operator $L_{\mathcal{M}}$ is not negative and it has a unique self-adjoint non-negative square root $L_{\mathcal{M}}^{1/2}$.
One can introduce another scale of Sobolev-Mellin spaces which are domains $\mathcal{D}(L_{\mathcal{M}}^{k/2}),\>k\in \mathbb{N},$ of powers of $L_{\mathcal{M}}^{1/2}$ with the graph norm:
$$
\|f\|_{\mathcal{D}(L_{\mathcal{M}}^{k/2})}=\|f\|_{X^{2}}+\|L_{\mathcal{M}}^{k/2} f\|_{X^{2}}.
$$
\begin{thm}\label{equivalence}
The spaces $\mathcal{D}(L_{\mathcal{M}}^{k/2})$ and $\mathcal{W}_{2}^{k},\>k\in \mathbb{N},$ coincide and their norms are equivalent.

\end{thm}

\begin{proof}
One has
$$
\sum_{j=1}^{n}\|Q_{j}f\|_{X^{2}}^{2}=\sum_{j=1}^{n}\langle \Theta_{j}f, \Theta_{j}f\rangle_{X^{2}}=\left\langle-\sum_{j=1}^{n}\Theta_{j}^{2}f, f\right\rangle_{X^{2}}=\left\langle   L_{\mathcal{M}}f, f       \right\rangle_{X^{2}}=\|L_{\mathcal{M}}^{1/2}f\|_{X^{2}}.
$$
Thus the spaces $\mathcal{D}(L_{\mathcal{M}}^{1/2})$ and $\mathcal{W}_{2}^{1}$ coincide and their norms equivalent.
Now, since $\Theta_{j}$ are commuting with each other, we have
$$
\sum_{1\leq k,j\leq n}\|\Theta_{k}\Theta_{j}f\|^{2}_{X^{2}}=\sum_{j=1}^{n}\|L_{\mathcal{M}}^{1/2}\Theta_{j}f\|_{X^{2}}^{2}=\sum_{j=1}^{n}\|\Theta_{j}L_{\mathcal{M}}^{1/2}f\|_{X^{2}}^{2}=\|L_{\mathcal{M}}f\|^{2}_{X^{2}},
$$
which proves the statement for $k=2$. Continue this way we can finish the proof. 
\end{proof}
We set 
$$
{\bf  w}^{r}_{2}( s, f; L_{\mathcal{M}})=\sup_{0\leq \tau\leq s}\|\left(e^{-i\tau L_{\mathcal{M}}}-I\right)^{r}f\|_{X^{2}}.
$$
The previous Theorem implies the following Corollary.

\begin{col}
The norm of the Besov-Mellin space $\mathcal{B}_{2,q}^{\alpha}, \>1\leq q\leq \infty, $ is equivalent to \begin{equation}\label{Bnorm30}
\|f\|_{X^{2}}+\left(\int_{0}^{\infty}(s^{-\alpha}{\bf w}^{r}_{2}(s,
f; L_{\mathcal{M}}))^{q} \frac{ds}{s}\right)^{1/q}.
\end{equation}
Moreover, for $\alpha$ not-integer it is equivalent to
\begin{equation}\label{Bnorm10}
\|f\|_{ \mathcal{D}\left(L_{\mathcal{M}}^{[\alpha]}\right) }
+
\left(\int_{0}^{\infty} \left(s^{[\alpha]-\alpha}{\bf w}^{1}_{2}
(s, L_{\mathcal{M}}^{[\alpha]}f; L_{\mathcal{M}})\right)^{q}\frac{ds}{s}\right)^{1/q}
\end{equation}
where $[\alpha]$ is the integer part of $\alpha$, and for  
$\alpha=k\in \mathbb{N}$  integer  its norm  is equivalent to the norm 
 (Zygmund condition)
\begin{equation}\label{Bnorm20}
\|f\|_{\mathcal{D}\left(L_{\mathcal{M}}^{k-1}\right)}+ 
\left(\int_{0}^{\infty}\left(s^{-1}{\bf w}^{2}_{2}(s,
L^{k-1}f; L_{\mathcal{M}})\right)
 ^{q}\frac{ds}{s}\right)^{1/q}.
\end{equation}

\end{col}

Note that according to the general theory (see for example \cite{KPS}, \cite{T}) one has the following isomorphisms for the self-adjoint operator $L^{1/2}_{\mathcal{M}}$ for $0<\alpha<r$:
$$
\mathcal{B}_{2,2}^{\alpha}=\left( X^{2}, \mathcal{W}_{2}^{\alpha}\right)_{\alpha/r,\>2}^{K}=\left(X^{2}, L_{\mathcal{M}}^{r/2}\right)^{K}_{\alpha/r,\>2}=\mathcal{D}\left(L_{\mathcal{M}}^{\alpha/2}\right).
$$
Thus when $q=2$ the norms (\ref{Bnorm30})-(\ref{Bnorm20}) are equivalent to the graph norm of a corresponding space $\mathcal{D}\left(L_{\mathcal{M}}^{\alpha/2}\right)$.

Let $\mathcal{D}^{\infty}(L_{\mathcal{M}}^{1/2})=\cap_{m\in \mathbb{N}}\mathcal{D}(L_{\mathcal{M}}^{m/2})$, where $\mathcal{D}(L_{\mathcal{M}}^{m/2})$  is the domain of $L_{\mathcal{M}}^{m/2}.$
One  has the following version of the Landau-Kolmogorov-Stein inequality which follows from Theorem \ref{KSM-Ap}.
\begin{thm} The following holds for $f\in \mathcal{D}(L_{\mathcal{M}}^{m/2})$
\begin{equation}
\left\|L_{\mathcal{M}}^{k/2}f\right\|_{X^{2}}^m \leq C_{k,m}\|L_{\mathcal{M}}^{m/2}f\|_{X^{2}}^{k}\|f\|_{X^{2}}^{m-k},\>\>\>
0\leq k \leq m,\label{KSM}
\end{equation}
where $C_{k,m}= (K_{m-k})^m/(K_{m})^{m-k},$ and $K_{j}$ is the Favard constant.
\end{thm}

\begin{defn} 
 Let $ \mathbf{B}^{2}_{\sigma}(L_{\mathcal{M}}^{1/2}),$ be the space of all $f\in X^{2}$ for which the following  inequality holds
\begin{equation}\label{BB}
\|L_{\mathcal{M}}^{k/2}f\|_{X^{2}}\leq \sigma^{k}\|f\|_{X^{2}},
\end{equation}
for all natural $k$. 
 For a function $f\in \cup_{\sigma>0}\mathbf{B}^{2}_{\sigma}(L_{\mathcal{M}}^{1/2}),$  let  the $\sigma_{ f}$ be the smallest $\sigma$ for which the  inequality  (\ref{BB}) holds.
 \end{defn}
 The next Theorem follows from Theorem \ref{new-Ap} in Appendix 2. 
\begin{thm}\label{new}

Let $f\in X^{2}$ belong to a space  $\mathbf{B}^{2}_{\sigma}(L_{\mathcal{M}}^{1/2}),$ for some
$0<\sigma<\infty.$ Then the following limit exists
 \begin{equation}
 d_{f}=\lim_{k\rightarrow \infty} \|L_{\mathcal{M}}^{k/2}
f\|^{1/k}_{X^{2}}, \label{limit}
\end{equation}
and $d_{f}=\sigma_{f}.$ 
Conversely, if
$f\in \mathcal{D}^{\infty}(L_{\mathcal{M}}^{1/2})$ and $d_{f}=\lim_{k\rightarrow \infty}
\|L_{\mathcal{M}}^{k/2} f\|^{1/k}_{X^{2}},$ exists and is finite, then $f\in\mathbf{B}^{2}_{d_{f}}(L_{\mathcal{M}}^{1/2})$
and $d_{f}=\sigma_{f }.$
\end{thm}

\section{Paley-Wiener-Mellin  spaces $PW_{\sigma}(L_{\mathcal{M}}^{1/2})$} \label{Hilbert}

\subsection{The spectral theorem approach}

Consider a self-adjoint non-negative definite operator
$L_{\mathcal{M}}$ in the Hilbert space $X^{2}(\mathbb{R}_{+}^{n}, d\mu)$ and
let $\sqrt{L_{\mathcal{M}}}$ be the non-negative square root of $L_{\mathcal{M}}$.

According to the spectral theory
for self-adjoint operators \cite{BS}
there exists a direct integral of
Hilbert spaces $H=\int H(\lambda )dm (\lambda )$ and a unitary
operator $\mathcal{F}$ from $X^{2}$ onto $H$, which
transforms the domains of $L_{\mathcal{M}}^{k/2}, k\in \mathbb{N},$
onto the sets
$H_{k}=\{x \in H|\lambda ^{k}x\in H \}$
with the norm 
\begin{equation}\label{FT}
 \|x\|_{H_{k}}= \left<x, x\right>^{1/2}_{H_{k}}=\left (\int^{\infty}_{0}
 \lambda^{2k}\|x(\lambda )\|^{2}_{H(\lambda)} dm
 (\lambda ) \right )^{1/2},
 \end{equation}
and satisfies the identity
$\mathcal{F}(L_{\mathcal{M}}^{k/2} f)(\lambda)=
 \lambda ^{k} (\mathcal{F}f)(\lambda), $ if $f$ belongs to the domain of
 $L_{\mathcal{M}}^{k/2}$.
 We call the operator $\mathcal{F}$ the Spectral Fourier Transform \cite{Pes00}. As known, $H$ is the set of all $m $-measurable
  functions $\lambda \mapsto x(\lambda )\in H(\lambda ) $,
  for which the following norm is finite:
$$\|x\|_{H}=
\left(\int ^{\infty }_{0}\|x(\lambda )\|^{2}_{H(\lambda )}dm
(\lambda ) \right)^{1/2}.
$$
For a function $F$ on $[0, \infty)$ which is bounded and measurable
with respect to $dm$
one can introduce the  operator $F(L_{\mathcal{M}})$ as a multiplication operator 
by using the formula
\begin{equation}\label{Op-function}
F(L_{\mathcal{M}}) f=\mathcal{F}^{-1}\>F\>\mathcal{F}f,\>\>\>f\in X^{2}.
\end{equation}
If $F$ is real-valued the operator $F(L_{\mathcal{M}})$ is self-adjoint.
\begin{defn}\label{PWvector}
We say that a function  $f \in X^{2}(\mathbb{R}_{+}^{n}, d\mu)$ belongs to the  Paley-Wiener-Mellin space $PW_{\sigma}\left(L_{\mathcal{M}}^{1/2}\right)$
if the support of the Spectral Fourier Transform
$\mathcal{F}f$ is contained in $[0, \sigma]$.
\end{defn}
The next two facts are obvious.
\begin{thm}The spaces $PW_{\sigma}\left(L_{\mathcal{M}}^{1/2}\right)$ have the following properties:
\begin{enumerate}
\item  the space $PW_{\sigma}\left(L_{\mathcal{M}}^{1/2}\right)$ is a linear closed subspace in
$X^{2}$,
\item the space 
 $\bigcup _{ \sigma >0}PW_{\sigma}\left(L_{\mathcal{M}}^{1/2}\right)$
 is dense in $X^{2}$.
\end{enumerate}
\end{thm}

The next theorem contains generalizations of several results
from  classical harmonic analysis (in particular  the
Paley-Wiener theorem). It follows from our  results in
\cite{Pes00} (see also Appendix 3).
\begin{thm}\label{PWproprties}
The following statements hold:
\begin{enumerate}
\item (Bernstein-Mellin inequality)   $f \in PW_{\sigma}\left(L_{\mathcal{M}}^{1/2}\right)$ if and only if
$ f \in \mathcal{D}^{\infty}(L_{\mathcal{M}})=\bigcap_{k=1}^{\infty}\mathcal{D}(L_{\mathcal{M}}^{k})$,
and the following Bernstein-type inequalities  holds true
\begin{equation}\label{Bern0}
\|L_{\mathcal{M}}^{s/2}f\|_{X^{2}} \leq \sigma^{s}\|f\|_{X^{2}}  \quad \mbox{for all} \, \,  s\in \mathbb{R}_{+};
\end{equation}
 \item  (Paley-Wiener-Mellin theorem) $f \in PW_{\sigma}(L_{\mathcal{M}}^{1/2})$
  if and only if for every $g\in X^{2}$ the scalar-valued function of the real variable  $ t \mapsto
\langle e^{-it\sqrt{L_{\mathcal{M}}}}f,g \rangle $
 is bounded on the real line and has an extension to the complex
plane as an entire function of the exponential type $\sigma$;
\item (Riesz-Boas-Mellin interpolation formula) $f \in PW_{\sigma}\left(L_{\mathcal{M}}^{1/2}\right)$ if
and only if  $ f \in \mathcal{D}^{\infty}(L_{\mathcal{M}})$ and the
following Riesz-Boas-Mellin interpolation formula holds for all $\sigma > 0$:
\begin{equation}  \label{Rieszn}
i\sqrt{L_{\mathcal{M}}}f=\frac{\sigma}{\pi^{2}}\sum_{k\in\mathbb{Z}}\frac{(-1)^{k-1}}{(k-1/2)^{2}}
e^{-i\left(\frac{\pi}{\omega}(k-1/2)\right)\sqrt{L_{\mathcal{M}}}}f.
\end{equation}
\end{enumerate}
\end{thm}
\begin{proof}
(1) follows immediately from the definition and representation (\ref{FT}).  To prove (2) it is sufficient to apply the classical Bernstein inequality \cite{N}  in the  uniform norm on $\mathbb{R}$ to every  function $\langle e^{-it\sqrt{L_{\mathcal{M}}}}f,g\rangle,\>\>g\in X^{2}$. To prove   (3) one has to apply the classical Riesz-Boas interpolation formula on $\mathbb{R}$ \cite{N},  \cite{Pes15}, to the function
$\langle e^{-it\sqrt{L_{\mathcal{M}}}}f,g\rangle$.
\end{proof}

The Bernstein inequality (\ref{Bern0}) and Theorem \ref{equivalence} imply the following.

\begin{thm}
The Bernstein-Mellin space $\cup_{\sigma>0}{\bf B}^{2}_{\sigma}(\Theta_{1}, ... ,\Theta_{n})$ and the Paley-Wiener-Mellin  space $\cup_{\sigma>0}PW_{\sigma}(L_{\mathcal{M}}^{1/2})$ coincide. 
\end{thm}
Thus the best approximation functional $\mathcal{E}_{2}(s,f)$ defined in (\ref{BestAp}) can be redefined as
$$
\mathcal{E}_{2}(\sigma, f)=\mathcal{E}_{2}(\sigma, f; L_{\mathcal{M}})=\inf_{g\in PW_{\sigma}(L_{\mathcal{M}}^{1/2})}\|f-g\|_{X^{2}}.
$$

\subsection{The Mellin transform approach}

For an approach to the Paley-Wiener-Mellin spaces in one-dimensional case  in terms of the Mellin transform we refer to the very interesting papers \cite{BBMS-5},  \cite{BBMS-6}, \cite{BBMS-8}. Note that in one-dimensional case this approach is simply identical to the approach described in the previous section.

Namely, it is known \cite{BJ1, BJ2} that the Mellin  transform 
\begin{equation}\label{M}
\mathcal{M}f(t)=\int_{0}^{\infty}f(x)x^{it}\frac{dx}{x},
\end{equation}
 which is originally defined for functions in $X^{1}(\mathbb{R}_{+}, d\mu)\cap X^{2}(\mathbb{R}_{+}, d\mu)$  can be extended to an isomorphism of $X^{2}(\mathbb{R}_{+}, d\mu)$ onto $L^{2}(i\mathbb{R}, dt)$ and the following relation holds
$$
\mathcal{M} \left[\Theta^{r}f\right](it)=(-it)^{r}\mathcal{M}f(t),\>\>\>f\in \mathcal{D}(\Theta^{r}),
$$
where $\Theta=x\frac{d}{dx}$ is the Mellin differentiation operator in the Hilbert space $X^{2}(\mathbb{R}_{+}, d\mu)$ and $ \mathcal{D}(\Theta^{r})$ is the domain of $\Theta^{r},\>r\in \mathbb{N}$.

It implies, that   the Mellin transform in $X^{2}(\mathbb{R}_{+}, d\mu)$ provides a digonalization of the self-adjoint, non-negative Laplace operator $L_{\mathcal{M}}=-\Theta^{2}$ in $X^{2}(\mathbb{R}_{+}, d\mu)$ in the sense that the following formula holds
\begin{equation}\label{M}
\mathcal{M}\left [L_{\mathcal{M}} f\right](t)=\mathcal{M} \left[\Theta^{2}f\right](t)=-t^{2}\mathcal{M}f(t),\>\>\>f\in \mathcal{D}(\Theta^{2}).
\end{equation}

In the $n$-dimensional case a similar approach can be implemented by using the $n$-dimensional version of the Mellin transform (see \cite{Tu}). However, this approach is slightly different from the approach of the previous subsection since it is relying on a different form of the Spectral Theorem.

One can show that for the transform 
$$
\mathcal{M}^n{}f=\int_{0}^{\infty}...\int_{0}^{\infty}f(x_{1}, ..., x_{n})x_{1}^{it_{1}}... x_{n}^{it_{n}}\frac{dx_{1}}{x_{1}}...\frac{dx_{n}}{x_{n}},
$$
 the following equality holds
$$
\mathcal{M}^{n}\left[L_{\mathcal{M}} f\right](t_{1}, ..., t_{n})=-(t^{2}_{1}+... +t_{n}^{2})\mathcal{M}^{n}f(t_{1}, ..., t_{n}),\>\>\>f\in X^{2}(\mathbb{R}_{+}^{n}, d\mu),
$$
where $-L_{\mathcal{M}}=\Theta_{1}^{2}+ ... +\Theta_{n}^{2}$. 

Next, one can define $\widetilde{PW_{\sigma}}\left(L_{\mathcal{M}}^{1/2}\right)$ as a subspace of functions $f \in X^{2}(\mathbb{R}_{+}^{n}, d\mu)$  for which the support  of $\mathcal{M}^{n}f$ is in the ball $\sqrt{t^{2}_{1}+... +t_{n}^{2}}\leq \sigma$. This definition is essentially equivalent to the our 
Definition \ref{PWvector}.

\section{Paley-Wiener-Mellin frames in $X^{2}(\mathbb{R}_{+}^{n}, d\mu)$.}\label{Frames}

The construction of frequency-localized frames is  achieved via spectral calculus. The idea is to start from a partition of unity on the positive real axis. In the following, we will be considering two different types of such partitions, whose construction we now describe in some detail. The construction below was discribed in \cite{FFP}.

 Let $g\in C^{\infty}(\mathbb{R}_{+})$ be a non-increasing
 function such that $\rm supp(g)\subset [0,\>  2], $ and $g(\lambda)=1$ for $\lambda\in [0,\>1], \>0\leq g(\lambda)\leq 1, \>\lambda>0.$
 We now let
$$
 h(\lambda) = g(\lambda) - g(2 \lambda).
$$
 The function $h$ is supported in $[2^{-1}, 2]$ we 
and use it to define
$$
 F_0(\lambda) = \sqrt{g(\lambda)}~, F_j(\lambda) = \sqrt{h(2^{-j} \lambda)}~, j \ge 1~,
$$
as well as $
 G_j(\lambda) = \left[F_j(\lambda)\right]^2=F_j^2(\lambda)~, j \ge 0~.$
As a result of the definitions, we get for all $\lambda \ge 0$ the equations
$$
\sum_{j =0} ^{k}G_j(\lambda) = \sum_{j =0}^{k} F_j^2(\lambda)
=  g(2^{-k} \lambda),
$$
and as a consequence
$$
\sum_{j \geq 0} G_j(\lambda) = \sum_{j \geq 0} F_j^2(\lambda) =  1~,\>\>\>\lambda\geq 0,
$$
 with finitely many nonzero terms occurring in the sums for each
 fixed $\lambda$. 
We call the sequence $(G_j)_{j \ge 0}$ a {\bf (dyadic) partition of unity}, and $(F_j)_{j \ge 0}$ a {\bf quadratic (dyadic) partition of unity}.  As will become soon apparent, quadratic partitions are useful for the construction of frames.
 Using the spectral theorem one has
$$
F_{j}^{2}(L_{\mathcal{M}})  f=\mathcal{F}^{-1}\>F_{j}^{2}\>\mathcal{F}f,\>\>\>j\geq 1,
$$
and thus
\begin{equation} \label{eqn:quad_part_identity}
 f = \mathcal{F}^{-1}\mathcal{F}f =\mathcal{F}^{-1}\left(  \sum_{j \geq 0}F_{j}^{2}\right) \mathcal{F}f=  \sum_{j\geq 0}  \left(\mathcal{F}^{-1}\>F_{j}^{2}\>\mathcal{F}f\right)=
 \sum_{j\geq 0}F_{j}^2(L_{\mathcal{M}}) f
\end{equation}  
Taking inner product with $f$ gives
$$
\|F_{j}(L_{\mathcal{M}}) f\|^{2}_{X^{2}}=\langle F_{j}^{2}(L_{\mathcal{M}}) f, f \rangle_{X^{2}},
$$  
and
\begin{equation}\label{Decomp}
\|f\|_{X^{2}}^2=\sum_{j \geq 0}\langle F_{j}^2 (L_{\mathcal{M}}) f,f\rangle_{X^{2}}=\sum_{j \geq 0}\|F_{j}(L_{\mathcal{M}})f\|_{X^{2}}^2 .
\end{equation}
Similarly, we get the identity  
$$
 \sum_{j \geq 0} G_j (L_{\mathcal{M}}) f = f~.
 $$
Moreover, since the functions $G_j,  F_{j}$, have their supports in  $
[2^{j-1},\>\>2^{j+1}]$, the elements $ F_{j}(L_{\mathcal{M}}) f $ and $G_j (L_{\mathcal{M}}) f$
 are bandlimited to  $[2^{j-1},\>\>2^{j+1}]$, whenever $j \ge 1$, and to $[0,2]$ for $j=0$.

\section{More about Besov-Mellin  spaces $\mathcal{ B}^{\sigma}_{2, q}$}\label{MoreBes}

\subsection{Besov-Mellin  spaces $\mathcal{ B}^{\sigma}_{2, q}$ in terms of approximations}
In this section we apply Theorems \ref{Direct} and \ref{Inverse} in Appendix 1 to a situation where ${\bf E}=X^{2},\>{\bf F}=\mathcal{W}_{2}^{r}, \>\mathcal{B}_{2,q}^{\alpha}=(X^{2}, \mathcal{W}_{2}^{r})^{K}_{\alpha/r, q},$ and $\mathcal{T}=\cup_{\omega>0}PW_{\omega}\left(L_{\mathcal{M}}^{1/2}\right)$  is the abelian additive group of the underlying  vector space, with the quasi-norm
$$
 \| f \|_{\mathcal{T}} = \inf \left \{ \omega'>0~: f \in PW_{\mathbf{\omega}'}\left(L_{\mathcal{M}}^{1/2}\right) \right\}~.
$$
In previous sections we already proved that in our case the assumptions of Theorems \ref{Direct} and \ref{Inverse} are satisfied. It allows us to formulate the following result.

\begin{thm} \label{approx}
For $\alpha>0, \>\>\>1\leq q\leq\infty,$ the norm of
 $\mathcal{B}_{2,q}^{\alpha}$,
  is equivalent to
\begin{equation}
\|f\|_{X^{2}}+\left(\sum_{j=0}^{\infty}\left(2^{j\alpha }\mathcal{E}_{2}(
2^{j}, f; L_{\mathcal{M}}\right)^{q}\right)^{1/q}.
\end{equation}
\label{maintheorem1}
\end{thm}

Let the functions $F_{j}$ be as in section \ref{Frames}.

\begin{thm}\label{projections}
For $\alpha>0, \>\>\>1\leq q\leq\infty,$ the norm of
 $\mathcal{B}_{2,q}^{\alpha}$,
  is equivalent to

\begin{equation}
f \mapsto \left(\sum_{j=0}^{\infty}\left(2^{j\alpha
}\left \|F_j(L_{\mathcal{M}}) f\right \|_{X^{2}}\right)^{q}\right)^{1/q},
\label{normequiv-1}
\end{equation}
  with the standard modifications for $q=\infty$.
\end{thm}

\begin{proof}

We obviously have
$$
\mathcal{E}_{2}(2^{l}, f; L_{\mathcal{M}})\leq \sum_{j> l} \left \|F_j (L_{\mathcal{M}}) f\right \|_{X^{2}}.
$$
By using a discrete version of Hardy's inequality \cite{BB} we obtain the estimate
\begin{equation} \label{direct}
\|f\|_{X^{2}}+\left(\sum_{l=0}^{\infty}\left(2^{l\alpha }\mathcal{E}_{2}(
2^{l}, f; L_{\mathcal{M}})\right)^{q}\right)^{1/q}\leq 
$$
$$
C \left(\sum_{j=0}^{\infty}\left(2^{j\alpha
}\left \|F_j(L_{\mathcal{M}}) f\right \|_{X^{2}}\right)^{q}\right)^{1/q}.
\end{equation}
Conversely,
 for any $g\in PW_{2^{j-1}}\left(L_{\mathcal{M}}^{1/2}\right)$ we have
$$
\left\|F_j(L_{\mathcal{M}}) f\right\|_{X^{2}}=\left\|F_{j}(L_{\mathcal{M}}) (f-g)\right\|_{X^{2}}\leq \|f-g\|_{X^{2}}.
$$
This implies the estimate
$$
\left\|F_j(L_{\mathcal{M}}) f\right\|_{X^{2}}\leq \mathcal{E}_{2}(2^{j-1}, f; L_{\mathcal{M}}),
$$
which shows that the inequality opposite to (\ref{direct}) holds.
 The proof is complete.
\end{proof}

 \subsection{Paley-Wiener-Mellin frames in  $X^{2}(\mathbb{R}_{+}^{n}, d\mu)$}\label{Hilb}
We consider the Laplace-Mellin operator $L_{\mathcal{M}}$ defined in (\ref{L}) in the Hilbert spaces $X^{2}(\mathbb{R}_{+}^{n}, d\mu)$.

\begin{defn}

The notation $PW_{[2^{j-1}, \>2^{j+1})}\left(L_{\mathcal{M}}^{1/2}\right),\>\>j\in \mathbb{N}$ will be used for the space of functions $f$ in $X^{2}$ whose Spectral Fourier Transform $\mathcal{F}f$ has support contained in $[2^{j-1},\>\>2^{j+1})$.

\end{defn}

For every $j\in \mathbb{N}$ let 
$$
\{\Phi_{k}^{j}\}_{k=1}^{K_{j}},\>\>\>\>\>\>\Phi_{k}^{j}\in PW_{[2^{j-1}, \>2^{j+1})}\left(L_{\mathcal{M}}^{1/2}\right),
$$
$$
K_{j}\in \mathbb{N}\cup \{\infty\},
$$ 
 be a frame in $PW_{[2^{j-1}, \>2^{j+1})}\left(L_{\mathcal{M}}^{1/2}\right)$ with the fixed constants $a,\>b$, i.e.
\begin{equation}\label{frame}
a\|f\|_{X^{2}}^2\leq \sum_{k=1}^{K_{j}}\left|\left< f, \Phi^{j}_{k}\right>_{X^{2}}\right|^{2}\leq b\|f\|_{X^{2}}^2,\>\>\>f \in PW_{[2^{j-1}, \>2^{j+1})}\left(L_{\mathcal{M}}^{1/2}\right).
\end{equation}

The formula (\ref{Decomp}) and the general theory of frames imply the following statement.

\begin{thm}

\begin{enumerate}
\item The set of functions $\{\Phi_{k}^{j}\}$ will be a frame in the entire space $X^{2}$ with the same frame constants $a$ and $b$, i.e.
\begin{equation}
a\|f\|_{X^{2}}^2\leq \sum_{j}\sum_{k}\left|\left< f, \Phi^{j}_{k}\right>_{X^{2}}\right|^{2}\leq b\|f\|_{X^{2}}^2,\>\>\>f \in X^{2}.
\end{equation}

\item  The canonical dual frame $\{\Psi^{j}_{k}\}$
also consists of bandlimited  vectors $\Psi^{j}_{k}\in  PW_{[2^{j-1},\>2^{j+1}]}(L_{\mathcal{M}}^{1/2}) ,\>\>j\in  \mathbb{N}, \>k=1,..., K_{j}$, and has the frame bounds $b^{-1},\>\>a^{-1}.$

\item The reconstruction formulas hold for every $f\in X^{2}$
$$
f=\sum_{j}\sum_{k}\left<f,\Phi^{j}_{k}\right>_{X^{2}}\Psi^{j}_{k}=\sum_{j}\sum_{k}\left<f,\Psi^{j}_{k}\right>_{X^{2}}\Phi^{j}_{k}.
$$

\end{enumerate}
\end{thm}

The formula (\ref{Decomp}) implies that in this case the set of functions $\{\Phi_{k}^{j}\}$ will be a frame in the entire $X^{2}$ with the same frame constants $a$ and $b$, i.e.
\begin{equation}
a\|f\|_{X^{2}}^2\leq \sum_{j}\sum_{k}\left|\left< f, \Phi^{j}_{k}\right>_{X^{2}}\right|^{2}\leq b\|f\|_{X^{2}}^2,\>\>\>f \in X^{2}.
\end{equation}

 \begin{thm}\label{framecoef}
For $\alpha>0,\>\>\> 1\leq q\leq\infty,$ the norm of
 $\mathcal{B}_{2,q}^{\alpha}$
  is equivalent to
\begin{equation}
 \left(\sum_{j=0}^{\infty}2^{j\alpha q }
\left(\sum_{k}\left|\left<          F_j(L_{\mathcal{M}}) f,                 \Phi^{j}_{k}\right>_{X^{2}}\right|^{2}\right)^{q/2}\right)^{1/q}\asymp \|f\|_{\mathcal{B}_{2,q}^{\alpha}},
\label{normequiv}
\end{equation}
  with the standard modifications for $q=\infty$.
\end{thm}
\begin{proof}
For   $f\in X^{2}$  and operator $F_{j}(L_{\mathcal{M}})$
we have  according to (\ref{frame})
\begin{equation}
a\left \|F_j (L_{\mathcal{M}}) f\right \|_{X^{2}}^{2}\leq
\sum_k\left|\left< F_{j} (L_{\mathcal{M}}) f, \Phi^{j}_{k}\right>_{X^{2}}\right|^{2}\leq b
\left\|F_j (L_{\mathcal{M}}) f\right \|_{X^{2}}^{2},
\end{equation}
 and then by means of (\ref{projections}) we obtain the statement.
 Theorem is proven.
\end{proof}

\section{Appendix 1.}\label{App1}

\subsection{$K$-functional and interpolation spaces}
\noindent
The goal of the section is to introduce basic notions of the theory of interpolation spaces  \cite{BB},  \cite{BL}, \cite{KPS}, \cite{T}, and  approximation spaces  \cite{BL}, \cite{BSch}, \cite{KP}, \cite{PS}.  
It is important to realize that the relations between
interpolation and approximation spaces cannot be described
in  the language of normed spaces. We have to make use of
quasi-normed linear spaces in order to treat them
simultaneously.

A quasi-norm $\|\cdot\|_{\bf E}$ on linear space $\bf E$ is
a real-valued function on $\bf E$ such that for any
$f,f_{1}, f_{2} \in \EB$ the following holds true:  
\begin{enumerate}

\item $\|f\|_{\bf E}\geq 0;\>\>\>$

\item $\|f\|_{\EB}=0  \Longleftrightarrow   f=0;\>\>\>$

\item $\|-f\|_{\EB}=\|f\|_{\EB};\>\>\>$

\item there exists some $C_{\EB} \geq 1$ such that
$\|f_{1}+f_{2}\|_{\EB}\leq C_{\EB}(\|f_{1}\|_{\EB}+\|f_{2}\|_{\EB}).\>\>$

\end{enumerate}
Two quasi-normed linear spaces $\EB$ and $\FB$ form a
pair if they are linear subspaces of a common linear space
$\AB$ and the conditions
$\|f_{k}-g\|_{\EB}\rightarrow 0,$ and
$\|f_{k}-h\|_{\FB}\rightarrow 0$
imply equality $g=h$ (in $\AB$).
For any such pair $\EB,\FB$ one can construct the
space $\EB \cap \FB$ with quasi-norm
$$
\|f\|_{\EB \cap \FB}=\max\left(\|f\|_{\EB},\|f\|_{\FB}\right)
$$
and the sum of the spaces,  $\EB + \FB$ consisting of all sums $f_0+f_1$ with $f_0 \in \EB, f_1 \in \FB$, and endowed with the quasi-norm
$$
\|f\|_{\EB + \FB}=\inf_{f=f_{0}+f_{1},f_{0}\in \EB, f_{1}\in
\FB}\left(\|f_{0}\|_{\EB}+\|f_{1}\|_{\FB}\right).
$$

Quasi-normed spaces $\HB$ with
$\EB \cap \FB \subset \HB \subset \EB + \FB$
are called intermediate between $\EB$ and $\FB$.
If both $E$ and $F$ are complete the
inclusion mappings are automatically continuous. 
An additive homomorphism $T: \EB \rightarrow \FB$
is called bounded if
$$
\|T\|=\sup_{f\in \EB,f\neq 0}\|Tf\|_{\FB}/\|f\|_{\EB}<\infty.
$$
An intermediate quasi-normed linear space $\HB$
interpolates between $\EB$ and $\FB$ if every bounded homomorphism $T:
\EB+\FB \rightarrow \EB + \FB$
which is a bounded homomorphism of $\EB$ into
$\EB$ and a bounded homomorphism of $\FB$ into $\FB$
is also a bounded homomorphism of $\HB$ into $\HB$.
On $\EB+\FB$ one considers the so-called Peetre's $K$-functional
$$
K(f, t)=K(f, t,\EB, \FB)=\inf_{f=f_{0}+f_{1},f_{0}\in \EB,
f_{1}\in \FB}\left(\|f_{0}\|_{\EB}+t\|f_{1}\|_{\FB}\right).\label{K}
$$
The quasi-normed linear space $(\EB,\FB)^{K}_{\theta,q}$,
with parameters $0<\theta<1, \,
0<q\leq \infty$,  or $0\leq\theta\leq 1, \, q= \infty$,
is introduced as the set of elements $f$ in $\EB+\FB$ for which
\begin{equation}
\|f\|_{\theta,q}=\left(\int_{0}^{\infty}
\left(t^{-\theta}K(f,t)\right)^{q}\frac{dt}{t}\right)^{1/q} < \infty .\label{Knorm}
\end{equation}

It turns out that $(\EB,\FB)^{K}_{\theta,q}$
with the quasi-norm
(\ref{Knorm})  interpolates between $\EB$ and $\FB$.

\subsection{Approximation spaces}

Let us introduce another functional on $\EB+\FB$,
where $\EB$ and $\FB$ form a pair of quasi-normed linear spaces
$$
\mathcal{E}(f, t)=\mathcal{E}(f, t, \mathbf{E},  \mathbf{F})=\inf_{g\in \FB,
\|g\|_{\FB}\leq t}\|f-g\|_{\EB}.
$$
\begin{defn}
The approximation space $\mathcal{E}_{\alpha,q}(\EB, \FB),
0<\alpha<\infty, 0<q\leq \infty $ is the quasi-normed linear spaces
of all $f\in \EB+\FB$ for which the quasi-norm
\begin{equation}
\| f \|_{\mathcal{E}_{\alpha,q}(\EB, \FB)} = \left(\int_{0}^{\infty}\left(t^{\alpha}\mathcal{E}(f,
t)\right)^{q}\frac{dt}{t}\right)^{1/q}
\end{equation}
is finite.
\end{defn}
The next two theorems   represent a very abstract version of what is  known as a Direct and an Inverse  Approximation Theorems \cite{PS, BL}. In the form it is stated below they were proved in \cite{KP}.

\begin{thm}\label{Direct}
 Suppose that $\mathcal{T}\subset \FB \subset \EB$ are quasi-normed
linear spaces and $\EB$ and $\FB$ are complete.
If there exist $C>0$ and $\beta >0$ such that
the following Jackson-type inequality is satisfied
$
t^{\beta}\mathcal{E}(t,f,\mathcal{T},\EB)\leq C\|f\|_{\FB},
\>\>t>0, \>\> f \in \FB,$
 then the following embedding holds true
\begin{equation}\label{imbd-1}
(\EB,\FB)^{K}_{\theta,q}\subset
\mathcal{E}_{\theta\beta,q}(\EB, \mathcal{T}), \quad \>0<\theta<1, \>0<q\leq \infty.
\end{equation}
\end{thm}

\begin{thm}\label{Inverse}
If there exist $C>0$ and $\beta>0$
such that
the following Bernstein-type inequality holds
$
\|f\|_{\FB}\leq C\|f\|^{\beta}_{\mathcal{T}}\|f\|_{\EB}
,\>\> f\in \mathcal{T},
$
then the following embedding holds true
\begin{equation}\label{imbd-2}
\mathcal{E}_{\theta\beta, q}(\EB, \mathcal{T})\subset
(\EB, \FB)^{K}_{\theta, q}  , \quad 0<\theta<1, \>0<q\leq \infty.
\end{equation}
\end{thm}

\subsection{Frames in Hilbert spaces}

A family of vectors $\{\Phi_{j}\}$  in a Hilbert space $\mathbf{H}$ is called a frame if there exist constants
$A, B>0$ such that
\begin{equation}
A\|f\|^{2}_{\mathbf{H}}\leq \sum_{j}\left|\left<f,\Phi_{j}\right>\right|^{2}    \leq B\|f\|_{\mathbf{H}} ^{2} \quad \mbox{for all} \quad f\in \mathbf{H}.
\end{equation}
The largest $A$ and smallest $B$ are called lower and upper frame bounds.

The family of scalars $\{\left<f,\Phi_{j}\right>\}$
represents a set of measurements of a vector $f$.
In order to reconstruct  the vector $f$
from this collection  of measurements in a linear way
one has to find another
(dual) frame $\{\Psi_{j}\}$.
Then a reconstruction formula is
$$
f=\sum_{j}\left<f,\Phi_{j}\right>\Psi_{j}.
$$
Dual frames are not unique in general.
Moreover it may be difficult to find a dual frame in concrete
situations.  
If $A=B=1$ the frame is said to be  tight   or Parseval.
Parseval frames are similar in many respects to orthonormal wavelet bases.  For example, if in addition all vectors $\Phi_{j}$ are unit vectors, then the frame is an  orthonormal basis.
The main feature of Parseval frames is that
decomposing
and synthesizing a vector from known data are tasks carried out with
the same family of functions, i.e., the Parseval frame is its own dual frame.

For more details on Interpolation and Approximation spaces see \cite{BB}, \cite{BL}, \cite{BSch}, \cite{KPS}, \cite{KP}, \cite{L-P}, \cite{PS}, \cite{ T}.

\section{Appendix 2. }\label{App2}

Let $D$ be a generator of a strongly continuous one parameter group $T(t),\>t\in \mathbb{R},$ of isometries of a Banach space $\mathbb{B}$ with the norm $\|\cdot \|$. Let $\mathcal{D}^{\infty}=\cap_{m\in \mathbb{N}}\mathcal{D}(D^{m})$, where $\mathcal{D}(D^{m})$  is the domain of $D^{m}.$ 
Let $\mathbb{B}_{\sigma}(D)$ be the space of all $f\in \mathbb{B}$ for which the Bernstein-type inequality holds, i.e. 
\begin{equation}\label{Bern000}
\|D^{k}f\|\leq \sigma^{k}\|f\|,\>k\in \mathbb{N}.
\end{equation}
The following statements \cite{Pes00}, \cite{Pes15},  imply  Lemma \ref{BernstCriterion} and Lemma \ref{BernstConstruction} respectively.

\begin{lem}\label{BernstCriterion-Ap}

A function $f\in \mathcal{D}^{\infty}(D)$ belongs to  $\mathbb{B}_{\sigma}(D), \>\>\sigma> 0 ,$ if and only if
 the quantity
\begin{equation}\label{Lem}
\sup_{k\in N} \sigma ^{-k}\|D^{k}f\|=R(f,\sigma)
\end{equation}
 is finite.
\end{lem}

\begin{thm}\label{BernstConstruction-Ap} If $T$ is a strongly continuous one parameter group $T(t),\>t\in \mathbb{R},$ of isometries generated by $D$ and  
 $\mu\in {\bf B}^{1}_{\sigma}(\mathbb{R}), \>\>\|\mu\|_{L_{1}(\mathbb{R})}=1,$ be an entire function of exponential
type $\sigma$ then for any $f\in \mathbb{B}$ the function
$$
P_{ \sigma}f=\int _{-\infty}^{\infty}\mu(t)T(t)fdt
$$
belongs to $\mathbb{B}_{\sigma}^{p}(D).$  The integral here is understood in the sense of Bochner for the abstract-valued function $\mu(t)T(t)f: \mathbb{R}\mapsto \mathbb{B}$.

\end{thm}

\subsection{Landau-Kolmogorov-Stein  inequality}

Let us introduce the Krein-Favard constants (see \cite{A}, Ch. V)
which are defined as
$$
K_{j}=\frac{4}{\pi}\sum_{r=0}^{\infty}\frac{(-1)^{r(j+1)}}{(2r+1)^{j+1}},\>\>\>
j,\>\>r\in \mathbb{N}.
$$
It is known, that the sequence of all Krein-Favard
constants with even indices is strictly increasing and belongs to
the interval $[1,4/ \pi)$ and the sequence of all Krein-Favard constants
with odd indices is strictly decreasing and belongs to the
interval $(\pi/4, \pi/2],$ i.e.,
\begin{equation}
K_{2j}\in [1,4/ \pi), \; K_{2j+1}\in (4/\pi, \pi/2],\label{Fprop}
\end{equation}
or
$$
1=K_{0}\leq K_{2}\leq ...  <4/\pi< ... \leq K_{3}\leq K_{1}= \pi/2.
$$
In what follows the constant
$$
C_{k,m}= (K_{m-k})^m/(K_{m})^{m-k}\leq (\pi/2)^{m}
$$
is playing an important role.
We are going to prove a generalization of the classical
Landau-Kolmogorov-Stein inequality \cite{L, Ko, Stein}. This  inequality was
first proved by Landau and Kolmogorov for $L^\infty (\mathbb{R})$ and later extended
to $L^p(\mathbb{R})$ for $1\leq p < \infty$ by Stein  (see also \cite{ Pes08}).

\begin{lem}\label{KSM-Ap} If $f\in \mathcal{D}(D^{m})$
 then
the following inequality holds
\begin{equation}\label{LKS}
\left\|D^{k}f\right\|^m \leq C_{k,m}\|D^{m}f\|^{k}\|f\|^{m-k},\>\>\>
0\leq k \leq m.
\end{equation}

\end{lem}
\begin{proof}
Indeed, for any $h^{*}\in \mathbb{B}^{*}$ the Kolmogorov inequality \cite {Stein}
applied to the entire function $\left<T(t)f,h^{*}\right>$ gives
$$
\left\|\left(\frac{d}{dt}\right)^{k}\left<T(t)f,h^{*}\right>\right\|^m_{
C(\mathbb{R}^{1})}\leq
C_{k,m}\left\|\left(\frac{d}{dt}\right)^{n}\left<T(t)f,h^{*}\right>\right\|_{
C(\mathbb{R}^{1})}^{k}\times
$$
$$
\left\|\left<T(t)f,h^{*}\right>\right\|_{C(\mathbb{R}^{1})}^{m-k},\quad
0<k< m,
$$
or
$$
\left\|\left<T(t)D^{k}f,h^{*}\right>\right\|_{ C(\mathbb{R}^{1})}^m\leq
C_{k,m}\left\|\left<T(t)D^{m}f,h^{*}\right>\right\|_{
C(\mathbb{R}^{1})}^{k}\left\|\left<T(t)f,h^{*}\right>\right\|_{C(\mathbb{R}^{1})}^{m-k}.
$$
We obtain
\begin{align*}
\left\|\left<T(t)D^{k}f,h^{*}\right>\right\|_{ C(\mathbb{R}^{1})}^m & \leq
C_{k,m}\|h^{*}\|^{k}\|D^{m}f\|^{k}\|h^{*}\|^{m-k}\|f\|^{m-k}\\
&\leq C_{k,m}\|h^{*}\|^m \|D^{m}f\|^{k}\|f\|^{m-k},
\end{align*}
which yields  when $t=0$ 
$$
\left|\left<D^{k}f,h^{*}\right>\right|^m\leq C_{k,m}\|h^{*}\|^m
\|D^{m}f\|^{k}\|f\|^{m-k}.
$$
By choosing $h$ such that $\left|<D^{k}f,h^{*}>\right|=\|D^{k}f\|$ and
$\|h^{*}\|=1$  we obtain (\ref{LKS}).
\end{proof}

\subsection{A characterization of  $\mathbb{B}_{\sigma}(D)$}

For a vector $f\in \cup_{\sigma>0}\mathbb{B}_{\sigma}(D),$ the $\sigma_{f}$ will denote the smalest $\sigma$ for which the inequality ({\ref{Bern000}) holds. 

\begin{thm}\label{new-Ap}
Let $f\in \mathbb{B}$ belong to a space  $\mathbb{B}_{\sigma}(D),$ for some
$0<\sigma<\infty.$ Then the following limit exists
 \begin{equation}
 d_f=\lim_{k\rightarrow \infty} \|D^k
f\|^{1/k}, \label{limit}
\end{equation}
and $d_f=\sigma_f.$ 
Conversely, if
$f\in \mathcal{D}^{\infty}$ and $d_f=\lim_{k\rightarrow \infty}
\|D^k f\|^{1/k},$ exists and is finite, then $f\in\mathbb{B}_{d_f}(D)$
and $d_f=\sigma_f .$
\end{thm}
\begin{proof}
According to Lemma \ref{LKS} we have $$ \left\|D^{k}f\right\|^n \leq
C_{k,m}\|D^{m}f\|^{k}\|f\|^{m-k}, \quad 0\leq k \leq m.
$$
Without loss of generality, let us assume that $\|f\|=1.$ Thus,
$$ \left\|D^{k}f\right\|^{1/k} \leq
(\pi/2)^{1/km}\|D^{m}f\|^{1/m}, \quad 0\leq k \leq m.
$$
Let $k$ be arbitrary but fixed. It follows that
$$\left\|D^{k}f\right\|^{1/k} \leq
(\pi/2)^{1/km}\|D^{m}f\|^{1/m}, \mbox{ for all } m\geq k,$$ which
implies that
$$\left\|D^{k}f\right\|^{1/k}\leq \underline{\lim}_{n\rightarrow
\infty}\|D^{m}f\|^{1/m}.$$ But \rm since this inequality is true for
all $k>0,$ we obtain that
$$\overline{\lim}_{k\rightarrow\infty}\|D^{k}f\|^{1/k}\leq
\underline{\lim}_{m\rightarrow \infty}\|D^{m}f\|^{1/m},$$ which
proves that $d_f=\lim_{k\rightarrow}\|D^{k}f\|^{1/k}$ exists.
Since $f\in \mathbb{B}_{\sigma}(D) $ the constant $\sigma_{f}$ is finite and we have 
$$
\|D^{k}f\|^{1/k}\leq \sigma_f  \|f\|^{1/k},
$$
 and by taking the
limit as $k\rightarrow\infty$ we obtain $d_f\leq \sigma_f.$ To show
that $d_f= \sigma_f,$ let us assume that $d_f< \sigma_f.$
Therefore, there exist $M>0$ and $\sigma $ such that $0<d_f<\sigma
< \sigma_f$ and $$\|D^k f\|\leq M \sigma^k, \quad \mbox{for all }
k>0 .$$ Thus, by Lemma  \ref{Lem-0} we have $f\in \mathbb{B}_\sigma(D) ,$
which is a contradiction to the definition of $\sigma_f.$

Conversely,  suppose that $d_f=\lim_{k\rightarrow \infty} \|D^k
f\|^{1/k}$ exists and is finite. Therefore, there exist $M>0$ and
$\sigma
>0$ such that $d_f<\sigma $ and
$$\|D^k f\|\leq M \sigma^k, \quad \mbox{for all } k>0 ,$$
which, in view of Lemma \ref{Lem-0}, implies that $f\in
\mathbb{B}_{\sigma}(D) 
.$ Now by repeating the argument in the first part of
the proof we obtain $d_f=\sigma_f ,$ where $\sigma_f=\inf\left\{
\sigma : f\in \mathbb{B}_{\sigma}(D)\right\}.$

Theorem \ref{new} is proven. 
\end{proof}

\section{Appendix 3}\label{App3}

The Appendix contains some comments on the proof of Theorem \ref{PWproprties}.
To prove the first item of it we proceed as follows. 
If  $f \in PW_{\sigma}\left(L_{\mathcal{M}}^{1/2}\right)$ then one has for $\mathcal{F}f$:
\begin{equation}\label{BBB}
\| L_{\mathcal{M}}^{s/2}f\|_{X^{2}}= \left(\int ^{\infty}_{0}\lambda
^{2s}\|\mathcal{F}f(\lambda)\|^{2}_{H(\lambda)}dm(\lambda
)\right)^{1/2}=
$$
$$
\left(\int^{\sigma}_{0}\lambda^{2k}\|\mathcal{F}f(\lambda
)\|^{2}_{H(\lambda)}dm(\lambda)\right)^{1/2} \leq \sigma
^{s}\|f\|_{X^{2}} , \>\>\>k\in \mathbb{R}_{+} , 
\end{equation}
 which gives Bernstein inequality for $f$.
Conversely, if $f$ satisfies Bernstein inequality then $x=\mathcal{F}f$ satisfies
$\|x\|^{2}_{H_{k}} \leq
\sigma^{2k}\|x\|_{H}^{2}.$ Suppose that there exists a set $\mu
\subset [ 0, \infty ]\setminus [ 0,
\mu]$ whose $m $-measure is not zero and $x|_{\mu }$ is not zero almost everywhere. We can
assume
that $\mu \subset
[\sigma +\epsilon , \infty )$ for some $\epsilon >0.$ Then for any $k\in
\mathbb{N}$ we have

$$ \int _{\mu }\|x(\lambda )\|^{2}_{H(\lambda)}dm (\lambda ) \leq \int
^{\infty }_{\sigma +\epsilon}\lambda ^{-2k}\|
\lambda ^{k}x(\lambda)\|^{2}_{ H(\lambda)}dm(\lambda)
\leq\|x\|^{2}_{H}\left(\frac{\sigma }{\sigma
+\epsilon} \right)^{2k}.
$$ 
Since $k$ here can be arbitrarily large, it implies  that either $x(\lambda)$ is zero on
$\mu $ or $\mu $ has measure
zero.  

\bigskip

To prove the second item of Theorem \ref{PWproprties}  one has to show  that the above Bernstein inequality (\ref{BBB}) implies that every function
\begin{equation}\label{last}
 t \mapsto
\langle e^{it\sqrt{L_{\mathcal{M}}}}f,g \rangle_{X^{2}}, \>\> \>t\in \mathbb{R}, \>\>f, g\in X^{2},
\end{equation}
is  a regular Bernstein function in ${\bf B}_{\sigma}^{\infty}(\mathbb{R})$ which is bounded on the real line. This fact implies its properties described in the second item.

\bigskip

To prove the third item of Theorem \ref{PWproprties} we use the fact that the function (\ref{last}) belongs to the regular space ${\bf B}_{\sigma}^{\infty}(\mathbb{R})$ and apply to it the classical Riesz-Boas interpolation formula (see \cite{Pes14}).

\section*{Acknowledgements}
I  thank the anonymous referees for their valuable and constructive comments, which helped improve the paper.

The author states that there is no conflict of interest.

\bibliographystyle{amsalpha}

\end{document}